\documentclass[11pt]{article}
	
	\newcommand{\blind}{0}
	
    \makeatletter
    \renewcommand\section{\@startsection {section}{1}{\z@}%
                                       {-3.5ex \@plus -1ex \@minus -.2ex}%
                                       {2.3ex \@plus.2ex}%
                                       {\normalfont\fontfamily{phv}\fontsize{16}{19}\bfseries}}
    \renewcommand\subsection{\@startsection{subsection}{2}{\z@}%
                                         {-3.25ex\@plus -1ex \@minus -.2ex}%
                                         {1.5ex \@plus .2ex}%
                                         {\normalfont\fontfamily{phv}\fontsize{14}{17}\bfseries}}
    \renewcommand\subsubsection{\@startsection{subsubsection}{3}{\z@}%
                                        {-3.25ex\@plus -1ex \@minus -.2ex}%
                                         {1.5ex \@plus .2ex}%
                                         {\normalfont\normalsize\fontfamily{phv}\fontsize{14}{17}\selectfont}}
    \makeatother
	
\usepackage{amssymb}
\usepackage{multirow}
\usepackage{booktabs}
\usepackage{amsmath}
\usepackage{adjustbox}
\usepackage{apacite}
\usepackage{algorithm}
\usepackage{algorithmic}
\usepackage{graphicx}
\usepackage{geometry}
\usepackage{subcaption}
\usepackage[figuresright]{rotating} 
\usepackage{natbib}
\usepackage{amsthm}
\newtheorem{theorem}{Theorem}

\newtheorem{prop}[theorem]{Proposition}
\newtheorem{obs}[theorem]{Observation}

\begin{document}
		
		\def\spacingset#1{\renewcommand{\baselinestretch}%
			{#1}\small\normalsize} \spacingset{1}
		
		\if0\blind
		{
			\title{\bf Use statistical analysis to approximate integrated order batching problem}
			\author{Sen Xue and Chuanhou Gao\\
			School of Mathematical Sciences, Zhejiang University, Hangzhou, China}
			\date{}
			\maketitle
		} \fi
		
		\if1\blind
		{

            \title{\bf \emph{IISE Transactions} \LaTeX \ Template}
			\author{Author information is purposely removed for double-blind review}
			
\bigskip
			\bigskip
			\bigskip
			\begin{center}
				{\LARGE\bf \emph{IISE Transactions} \LaTeX \ Template}
			\end{center}
			\medskip
		} \fi
		\bigskip

\begin{abstract}
Order picking and order packing entail retrieving items from storage and packaging them according to customer requests. These activities have always been the main concerns of the companies in reducing warehouse management costs. This paper proposes and investigates the Order Batching and Order Packing Problem, which considers these activities jointly. The authors propose a novel statistic-based framework, namely, the Max Correlation Reformulation problem, to find an approximation mixed-integer programming model. An approximation model is found within this framework in two phases. A lower dimension model is firstly proposed. Efforts are then made to increase its correlation coefficient with the original formulation. Finally, a powerful pairs swapping heuristics is combined with the approximation model. Numerical experiments show that this newly found approach outperforms the mainstream methods from the literature. It is demonstrated that this proposed method could significantly reduce the cost of picking and packing operations in a warehouse.
\end{abstract}

	\noindent%

	\spacingset{1.5} 

\maketitle

\section{Introduction} \label{s:intro}
Warehouses are a vital component of any supply chain (SC) system. They serve as receiving ends, storage and order picking locations and shipping departure points in SC systems. Moreover, today's companies are encountering competition in SC systems to meet customers' various needs. According to \cite{manzini2012warehousing}, the competition is based on time, cost, and customer service level. Warehouses, as a logistics transfer point and cost-intensive part, could largely determine the SC's overall performance. As a result, warehouse management is one of the most important aspects of an SC.

Warehouse management encompasses a series of non-independent product handling operations, including receiving, storage, order picking, packing, and shipping \citep{hompel2006warehouse}. The receiving, storage, and shipping activities are connected with logistics and warehouse design. The order picking process, defined as retrieving products from storage locations, is the most labor-intensive and costly procedure. According to \cite{bartholdi2014warehouse}, the picking consumes around 55\% of the total warehouse cost. Following close to the allocation of orders, the packing process, which also involves manual labor, concerns the packaging and sortation of products. This operation sequence indicates that picking and packing are strongly associated and should be jointly considered. As \cite{ackerman2012practical} mentioned, considering picking and packing together would provide a more accurate description of warehouse operations and avoid extra storage buffers.

The order batching problem (OBP) occurs in the order picking process when there exists a capacity constraint in picking facilities. Vast pieces of literature have investigated the OBP in picking separately. The cost function, which needs to be minimized, is mostly considered as the travel distance of item pickers in a picker-to-parts system, in which the item picker moves to storage locations to pick up items. The picker-to-parts system is more widespread than the part-to-picker system, in which the items are automatically retrieved to the picker, and is also considered in this paper.

Comparatively, a limited number of publications consider the OBP in the packing process and its integrated form with other activities. \cite{shiau2010warehouse} proposed a hybrid algorithm for picking and packing operations to generate a picking sequence. This algorithm aims to eliminate storage buffers and reduce total operation time. \cite{de2012determining} provided an integer-programming model in which they reduce the total picking and packing time by determining the optimal number of zones (i.e. a picking area in a warehouse assigned to a single item picker). \cite{boysen2018optimizing} devised an elementary optimization problem to shorten the packing time by minimizing the spread of orders in the release sequence. All the studies above are based on real-world cases and deeply related to problem characteristics. We also give an order batching problem integrated with the packing process, based on a practical warehouse case of one of China's largest logistics companies.

The OBP has a rather vital characteristic in practice: it is repeatedly solved corresponding to varying customer demands but in the same problem environment (e.g. the warehouse layout and the item storage location). This specialty inspires us to employ statistical tools to exploit it and develop approximation algorithms. Some efforts have been made in machine learning to learn features of optimization problems that have the same structure but different parameters \citep{bertsimas2021voice}. However, to the best of our knowledge, no similar work has been done on the OBP. 

\subsection{\emph{Related work}}
Until now, the majority of OBP solution approaches have been heuristics and metaheuristics, which are widely employed in practice. Only a few works attempt to solve OBP to or near optimality. There are also limited number of methods for combining statistics and data mining principles. One main reason is that the industry emphasizes the timeliness of an algorithm much more than the accuracy of the solution. The computing time of the B\&B algorithm makes it unpractical for medium or large-scale order batching problems. However, in this paper, we give a framework to find approximation models and use B\&B algorithms to get near-optimal solutions. To the best of our knowledge, no similar framework has been proposed in this area.

Heuristic algorithms consider defined rules for assigning orders. \cite{gibson1992order} introduced the First-Come-First-Served (FCFS) rule for OBP, which is most straightforward.  \cite{elsayed1981algorithms,elsayed1989order} firstly developed the seed algorithm. The seed algorithm uses selection rules to choose the starting order for each batch, then uses addition rules to add the remaining orders. The rules are formulated based on the problem characteristics. The saving algorithm, which was initially proposed by \cite{clarke1964scheduling} for the Vehicle Routing Problem, has evolved into many OBP variants, among which C\&W(i) and  C\&W(ii) are mostly mentioned. 

Metaheuristics are mostly common algorithms modified to OBP. \cite{albareda2009variable} developed a variable neighborhood search (VNS) based approach with four different warehouse configurations. \cite{henn2010metaheuristics} presented two methods, an iterated local search (ILS) and a rank-based ant system. \cite{henn2012tabu} proposed tabu search (TS) to minimize the total length of picking routes. 

\cite{tang2011combination} proposed a Lagrangian relaxation and column generation-based MIP approach. Following that, \cite{muter2015exact} developed a column generation-based exact method. The disadvantage of exact methods is that they cannot be extended to larger application instances (e.g. more than 100 orders). However, in this paper, we improve upon this situation by introducing an approximation MILP approach.

\cite{chen2005association} produced the first method combining data mining and MIP for OBP. They introduced the association rule mining to OBP and developed a MIP for maximizing order association in a batch and indirectly obtaining high-quality solutions with small total distances. \cite{ho2006study} proposed several batching methods established on seed-order selection rules (SOSR) and accompanying-order selection rules (AOSR). A subsequent study by \cite{ho2008order} on this topic gave new rules and interactions between SOSR and AOSR.

\subsection{\emph{Contributions and Paper Structure}}

In this paper, our contributions are as follows:

1. We propose and formulate a new order batching and order packing problem (OBOPP). 

2. We present a novel statistical-based framework to reformulate the mixed-integer programs(MIP) into easy-to-solve approximation mixed-integer linear programs (MILP). 

3. We provide two directions for reformulation: to decrease the scale and to increase the correlation value. 

4. We develop a problem-adapted order pairs swapping heuristic to further strengthen our approach. 

5. We conduct comparison tests to show the timeliness and accuracy of our algorithm. 

This paper is organized as follows. Section \ref{des} presents a description and model formulations of the OBOPP. Section \ref{sat} introduces a statistics-based framework to maintain an approximation MILP for OBOPP. Section \ref{sol} shows a heuristical strategy to obtain the approximation MILP. Section \ref{int} gives an iterated local search method for reinforcement. Section \ref{exp} examines our methods with the exact B\&B algorithm and classic heuristics of OBP. Finally, Section \ref{con} concludes the study and gives future research directions.

\section{The Order Batching and Order Packing problem}
\label{des}
\subsection{\emph{Problem Description}}

In this subsection, we describe the order batching and order packing problem (OBOPP). A warehouse is composed of parallel aisles and a packing area with a fixed number of packers. The packing area is in the middle right of the last aisle (See Figure \ref{layout}). This layout is similar to \cite{de2012determining}. In this study, the aisles are lined in length-descending order, from the left to the right side of the warehouse. This is a small traditional layout modification that arranges frequently ordered item classes to be stored on short aisles, leading to generally lower picking route distances. Our model considering this specialty can also be compatible with a traditional layout. For the relationship between orders and items, we follow the environment set by \cite{de1999efficient}. Each item has a unique storage location. An order, which represents a requirement of our customers, contains one or more items. The same items may appear in multiple orders at the same time.

The OBOPP can be described as a problem of finding a batching combination for the least cost of traveling time and packing time. At the beginning of each batch, the item-pickers depart from the depot point, travel down the aisles that index the locations of items, and collect the items. Then, the item-pickers carry the collected items to the packing point, where they are registered and sorted into orders. The time cost of traveling is the total walking length divided by the average speed. The time cost of registering and sorting is proportional to the number of unique items in practice. This is coherent because if all of the items are the same, the packing procedure can be completed with just one registration and partition by each order number. For multiple type items, the registration and partition processes are needed for each item type. Finally, we use the total traveling distance and the sum of unique items in each batch, multiplied by their respective weights, to represent the overall time cost. 

The batch size is limited by the packing point capability (i.e. a fixed number of packers) for packing orders to maintain a steady order flow from picking to packing, implying that each batch has a fixed number of orders. In this study, the capacity of the pickers' basket is always larger than the sum of items from the batch containing the largest orders. As a result, we consider the item-pick capacity to be a redundant restriction and omit it in our model. 

The manual pickers use the return strategy and the S-shape strategy for routing in practice. We follow the definitions of \cite{gu2007research} and \cite{de2007design}. In OBOPP, these strategies are characterized as follows (See Figure \ref{layout} for illustration). First, the order picker can only move in vertical or horizontal directions. For the return strategy, the order picker enters an aisle containing one or more items to be picked. He/she walks down the aisle until the final item required is picked, then returns to the bottom line to the next aisle. The S-shape strategy provides a solution in which the order picker traverses through an aisle containing one or more items to be collected. Then, he/she moves horizontally to the next aisle. If an item picker has completed all the items in a batch, he/she will move to the packing area to unload the items.

The OBOPP is NP-hard. The general OBP has been proven to be NP-hard \cite{gademann2005order}. In an OBOPP, if we set every item to be the same type, the packing time cost becomes constant, and the problem degenerates into a generic OBP. Thus, the OBOPP is at least as hard as OBP and then proven to be NP-hard.
\begin{figure}[htb]
	\centering
	\begin{minipage}{.5\textwidth}
		\centering
		\includegraphics[scale=0.18]{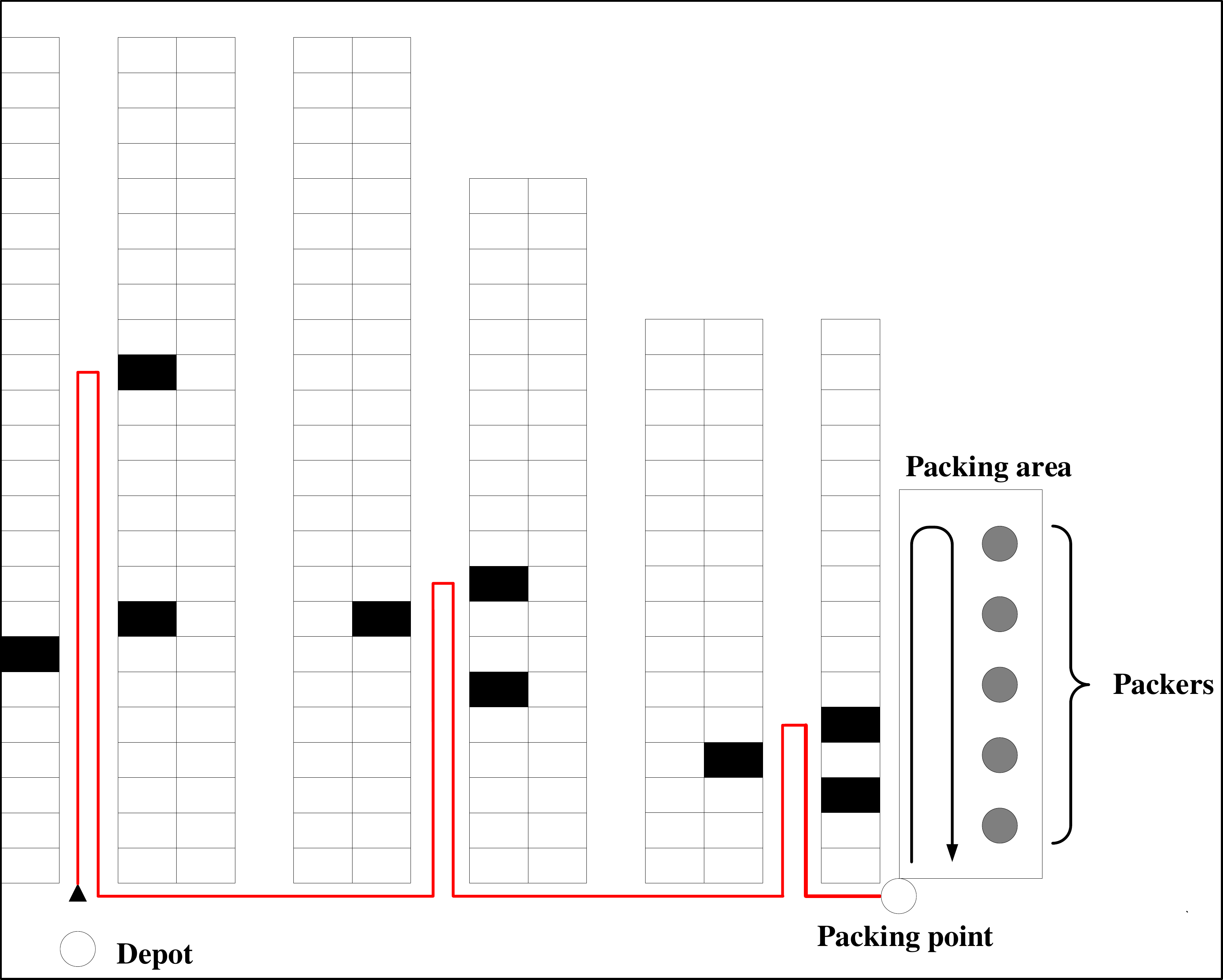}
		\label{rReturn}
	\end{minipage}%
	\begin{minipage}{.5\textwidth}
		\centering
		\includegraphics[scale=0.18]{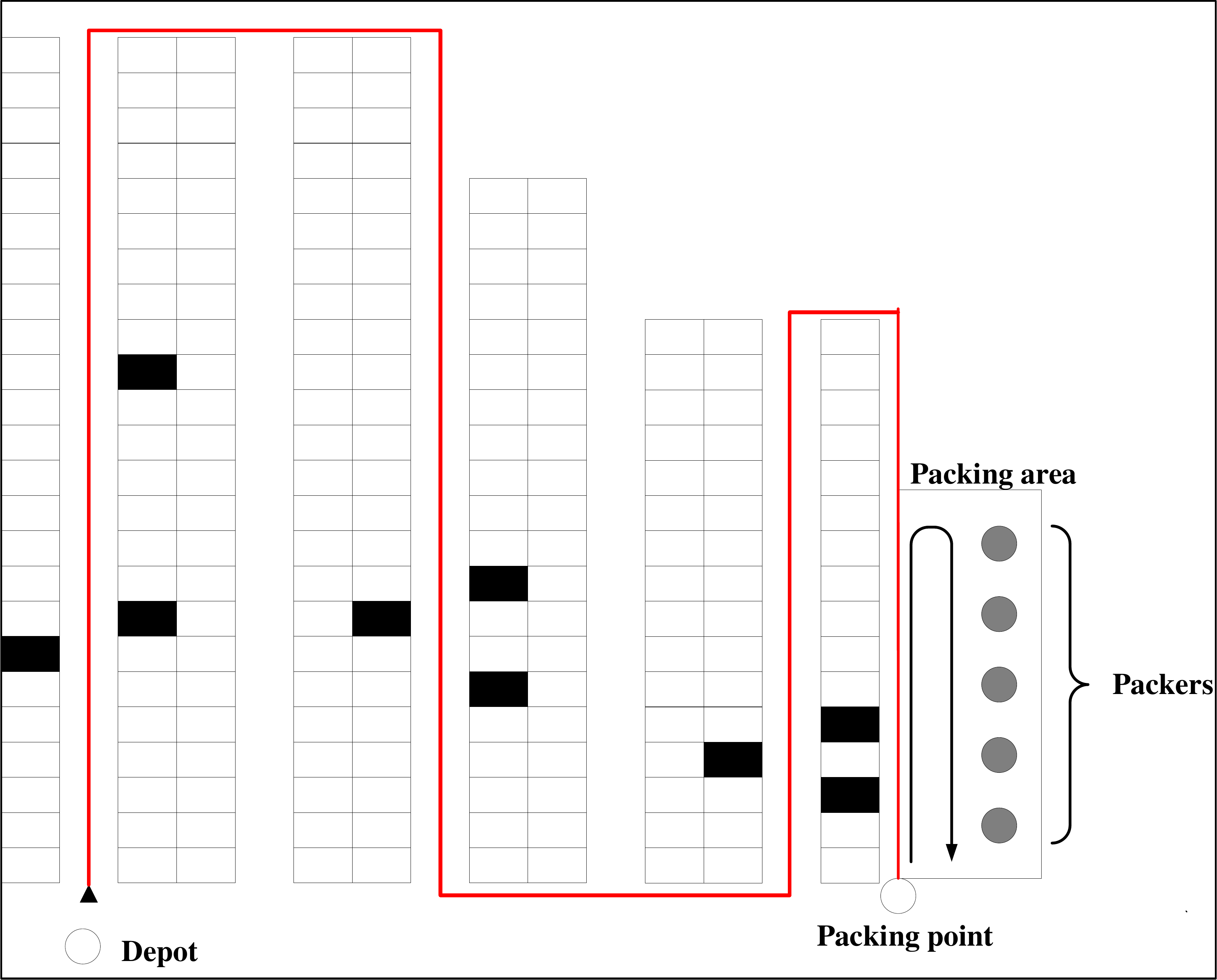}
		\label{rSshape}
	\end{minipage}
\caption{Example of Return routing strategy (left) and S-shape routing strategy (right) in the warehouse layout}
\label{layout}
\end{figure}
\subsection{\emph{Model Formulation}}

For the OBOPP defined above,  We give two MIP models corresponding to the return strategy and the S-shape strategy. The models' notations are listed below:

\quad \\

Parameters and indices:

\begin{tabular}{ll}
\hline

$I$, $i$ & the set of custom orders and its index $i\in I$;\\
$K$, $k$& the set of items(i.e. the set of items that appears in any of the orders) and its index $k \in K$;\\
$B$, $b$& the set of aisles and its index $b \in B$;\\
$J$, $j$& the set of feasible batches and its index $j \in J$;\\
$L$& the set of lengths of the aisles, where L=$\{l_1, l_2,...,l_{|B|}\}$;\\
$I_k$& the set of orders that include item k;\\
$I_b$& the set of orders that include one or more item that located in aisle b;\\
$r_{ib}$& $\max \{$t $\mid$ all items in order i that located at (b,t)$\}$;\\
$C$& capacity of the packing point;\\
$\tau$& weight of order packing time-consuming.\\
\hline
\end{tabular}

\quad \\

Variables:

$y_{ij}=\left\{\begin{array}{ll}1, & \text { if order $i$ is assigned to batch $j$  } \\ 0, & \text { otherwise }\end{array}\right.$

$z_{jk}=\left\{\begin{array}{ll}1, & \text { if item $k$ is in an order that assigned to batch $j$ } \\ 0, & \text { otherwise }\end{array}\right.$

$\delta_{jb}=\left\{\begin{array}{ll}1, & \text { if there exists item $k$ that located on aisle $b$ is assigned to batch $j$ } \\ 0, & \text { otherwise }\end{array}\right.$

$d_{jb}$=$\max \{ r_{ib} | y_{ij}=1\}$

$q_{jb}$=$\left\{\begin{array}{ll}n, & \text {if aisle $b$ is the nth aisle passed in batch j} \\ 0, & \text { otherwise }\end{array}\right.$

$o_{jb}$=$\left\{\begin{array}{ll}1, & \text {if aisle $b$ is the nth aisle passed in batch j and n is an odd number} \\ 0, & \text { otherwise }\end{array}\right.$

$u_{jb}$: an auxiliary variable to decide the parity of $w_{jb}$, which satisfies $q_{jb}=2 u_{jb}+o_{jb}$

The OBOPP with the return strategy can be stated as the following MIP:

\begin{equation}
    \min \{\tau \sum_{j \in J} \sum_{k \in K} z_{jk}+(1-\tau)  \sum _{j \in J} \sum_{b \in B}2 d_{jb} : (y,z,d) \in P^R _{IJKB}\}
    \label{eq0}
\end{equation}

where polytope $P^R _{IJKB}$ is given by: 
\begin{align}
    \sum_{j \in J} y_{ij} &=1, \quad \forall i \in I;
        \label{eq1}\\
    \sum_{i \in I} y_{ij} &= C, \quad \forall j \in J;
        \label{eq2}\\
    M_k z_{j,k} &\geq \sum_{i \in I_k} y_{ij}, \quad \forall j \in J, \quad \forall k \in K;
        \label{eq3}\\
     d_{jb} &\geq  r_{ib} y_{ij}, \quad \forall j \in J, \quad \forall b \in B;\label{eq4}\\
    y_{ij} &\in \{0,1\},\quad \forall i \in I,\quad \forall j \in J;\label{yd}\\
    z_{j,k} &\in \{0,1\},\quad \forall j \in J, \quad \forall k \in K;\label{zd}\\
    d_{jb} &\in R_+, \quad \forall j \in J, \quad \forall b \in B.\label{db}
\end{align}

The objective function in ($\ref{eq0}$) aims to minimize the total number of unique items in every batch and the total distance of the item-picker's picking path. The weight number $\tau$ represents how much percentage of all time cost is used in order packing procedure measured by the number of unique items. The constraint ($\ref{eq1}$) ensures that each order is assigned to only one batch. The constraint ($\ref{eq2}$) guarantees that each batch does not exceed its capacity limit. The big-M constraint ($\ref{eq3}$) is used to ensure that $z_{jk}$ indicates whether item $k$ is included in batch $j$. The constraint ($\ref{eq4}$) provides the walking length $d_{jb}$ for batch $j$ within aisle $b$. At last, constraints ($\ref{yd}$)-($\ref{db}$) define the domain of decision variables. 

The OBOPP with the S-shape strategy are stated as follows:

\begin{equation}
    \min \{\tau \sum_{j \in J} \sum_{k \in K} z_{jk}+(1-\tau) \sum _{j \in J} \sum_{b \in B} 2 l_b o_{jb}: (y,z,\delta,q,o,u)\in P^S _{IJKB}\}
    \label{objs}
\end{equation}
where polytope $P^S _{IJKB}$ is given by:
\begin{align}
    (\ref{eq1})&-(\ref{eq3});\\
    M_b \delta_{jb} &\geq \sum_{i \in I_b} y_{ij}, \quad \forall j \in J, \quad \forall b \in B;\label{Ms}\\
	q_{jb} &\leq \sum_{b \in B} \delta_{jb} , \quad \forall j \in J, \quad \forall b \in B;\label{upperboundw}\\
	q_{jb} &\geq \delta_{jb} , \quad \forall j \in J, \quad \forall b \in B;\label{lowerboundw}\\
	q_{j(b+1)}-q_{jb}&=q_{j(b+1)}(1-\delta_{jb})+\delta_{jb},\quad \forall j \in J, \quad \forall b \in B\backslash\{b_{|B|}\};\label{middleboundw}\\
	q_{jb}&=2u_{jb}+o_{jb}\quad \forall j \in J, \quad \forall b \in B;\label{oddbound}\\
	(\ref{yd})&-(\ref{zd});\\
    \delta_{jb} &\in \{0,1\},\quad \forall j \in J, \quad \forall b \in B;\label{deltad}\\
	q_{jb} &\in Z,\quad \forall j \in J, \quad \forall b \in B;\label{q}\\
	u_{jb} &\in Z,\quad \forall j \in J, \quad \forall b \in B;\label{u}\\
	o_{jb} &\in  \{0,1\},\quad \forall j \in J, \quad \forall b \in B.\label{o}
\end{align}
The decision variable  $\delta_{jb}$ in constraint ($\ref{Ms}$) determines whether a aisle $b$ is passed. Constraints ($\ref{upperboundw}$) and ($\ref{lowerboundw}$) give upper bound and lower bound for variable $q_{jb}$. Non-linear constraint ($\ref{middleboundw}$) guarantees that all the non-zero ordinal number $q_{jb}$ of a batch j is an arithmetic sequence with a difference of one. Constraint ($\ref{oddbound}$) ensures that $o_{jb}$ indicates whether or not $q_{jb}$ is an odd number. Constraint ($\ref{deltad}$)-($\ref{o}$) gives domains of variables $\delta_{jb}$, $q_{jb}$, $u_{jb}$ and $o_{jb}$.

The dimensions of conv($P^R _{IJKB}$) and conv($P^S _{IJKB}$) are large. The increment of dimensions is significantly contributed by adding auxiliary variables. This causes the classic B\&B method's search time to explode, rendering it impractical for larger cases (e.g. more than 100 orders). 

However, only variable $y_{ij}$ determines the outcome of this order batching problem,i.e., partitioning orders into separate batches. Let $M_{IJ}$ be the polytope on the partition of orders.

\begin{equation}
	M_{IJ}=\{ y_{ij} \in \{0,1\} :  \sum_{j \in J} y_{ij} =1,   \sum_{i \in I} y_{ij} = C, i \in I, j \in J  \}.
\end{equation} 

The relationships between $y_{ij}$ and other variables can be expressed by projection $Proj_{y} (P^R _{IJKB})=Proj_{y}(P^S _{IJKB})=M_{IJ}$. This provides the inspiration that we can simplify the model by removing undecided variables while ensuring that the model still preserves the problem structure. This is the driving force behind our statistics-based approach.

\section{Statistics-based approach for Model Reduction}
\label{sat}
\subsection{\emph{Analysis of Probability distribution}}
\label{sectionnormal}

In OBOPP, we consider the statistic of the solution values in a single problem with a data set. We first define the statistic problem as sampling the objective value of a set of random solutions. By random solution, we mean a $|J|$-partitioned permutation 
$$p=(\{a^k_1\}_{k \in [C]}, \{a^k_2\} _{k \in [C]},...,\{a^k _{|J|}\} _{k \in [C] }), \quad a^k _j \in I \quad \text{for all} \ j \in J \ \text{and} \ k \in [C],$$
where $ [C]=\{1,2,...,C\}$ and $\{a^k_j\}_{k \in [C]}$ for a number j is a set orders assigned to batch j. The sets of orders satisfy $\cup_{j \in J} \{a ^k _j\}_{k \in [C]} = I$, $\{a ^k _i\}_{k \in [C]} \cap \{a ^k _j\}_{k \in [C]} = \emptyset $ for a unequal pair of bathes $i,j \in J$. The $|J|$-partitioned permutation $p$ belongs to a probability space defined by $(\Omega, \mathcal{F}, P)$. The sample space $\Omega$ is the set of all feasible $|J|$-partitioned permutations $p$ and P$(p) = \frac{1}{|\Omega|}$.

\begin{obs}
	Given the set $I$, the feasible $|J|$-partitioned permutations are in one-to-one correspondence to the feasible integer points of $M_{IJ}$.
	\label{obs1}
\end{obs}

From the observation \ref{obs1} above, we notice that the random solution are equivalent to all the feasible solutions. To construct a random variable for sampling, for a $|J|$-partitioned permutation $\hat p$, we compute the objective values of the corresponding feasible integer point $\hat y(\hat p)$ by solving
\begin{equation}
\min \{\tau \sum_{j \in J} \sum_{k \in K} z_{jk}+(1-\tau)  \sum _{j \in J} \sum_{b \in B}2 d_{jb} : (\hat y (\hat p),z,d) \in P^R _{IJKB}\},
\label{return_fix}
\end{equation}
and
\begin{equation}
\min \{\tau \sum_{j \in J} \sum_{k \in K} z_{jk}+(1-\tau) \sum _{j \in J} \sum_{b \in B} l_b \delta_{jb}: (\hat y ( \hat p),z,\delta,q,u,o)\in P^S _{IJKB}\}.	
\label{Sshape_fix}
\end{equation}
corresponding to two routing strategies.

Because of the projection relationship mentioned above, the optimal solution exists. For simplification, we denote the objective function and polytope as a tuple $\pi = (f,P)$. Let random variable $X^{\pi}(p)$ be a mapping
$$  p \mapsto   \min \{f(y(p),w): (y(p),w)\in P\},$$

where $w$ indicates other variables except for variable $y$. The polytope $P$ satisfies $Proj _{y}=M_{IJ}$. To simplify the notions, we use $X^R$ and $X^S$ to denote the random variables of the initial objective functions and polytopes, where $R$ and $S$ are the tuples of objective functions and polytopes in \ref{return_fix} and \ref{Sshape_fix}. By generating random solution $p$, we can record the sampled value of $X^R$ and $X^S$ and further infer the distribution of them. Denote $X^{R} _{i}$ and $X^{S} _{i} $the $i$th sampled value. The problem becomes to infer the distribution of $X^R$ and $X^S$ in probability space $(\Omega, \mathcal{F}, P)$. 

The experiments are organized as a Monte Carlo Sampling study. We randomly generate $n$ random solutions $p$ with a real-world data instance (800 orders). The method of sampling from random variable $X^{\pi}(p)$ can be expressed as the following three steps:

1. Generate a random permutation of elements in $I$ and partition the permutation into $|J|$ feasible batches to maintain $\hat{p}$. Get the corresponding feasible integer point $\hat y(\hat{p})$. Go to step 2

2. Solve the problem $\min \{f(\hat y(\hat{p}),w): (\hat y(\hat{p}),w)\in P\}$ with $\hat{y}(\hat{p})$. Record the objective value in a list.

3. Repeat step 1 until the list has $n$ records.

The generating process in step 1 ensures the equivalent probability of each solution in the sample space $\Omega$. We use the Python packages   \texttt{Statistics} and \texttt{Pingouin} for normality tests. We set $n=5000$ to ensure significant conclusions of the experiments.

The Experiment results are shown below. We plot empirical probability density functions (PDF) (See Figure \ref{pdReturn} and \ref{pdSshape}) with n sampled points to show the probability density of $X^R$ and $X^S$. To verify the normality, we conduct a Kolmogorov-Smirnov (K-S) test and the results accept the normality hypothesis with p-values equal to 0.8453 and 0.8236.

\begin{figure}
\centering
\begin{subfigure}{.5\textwidth}
  \centering
		\includegraphics[scale=0.5]{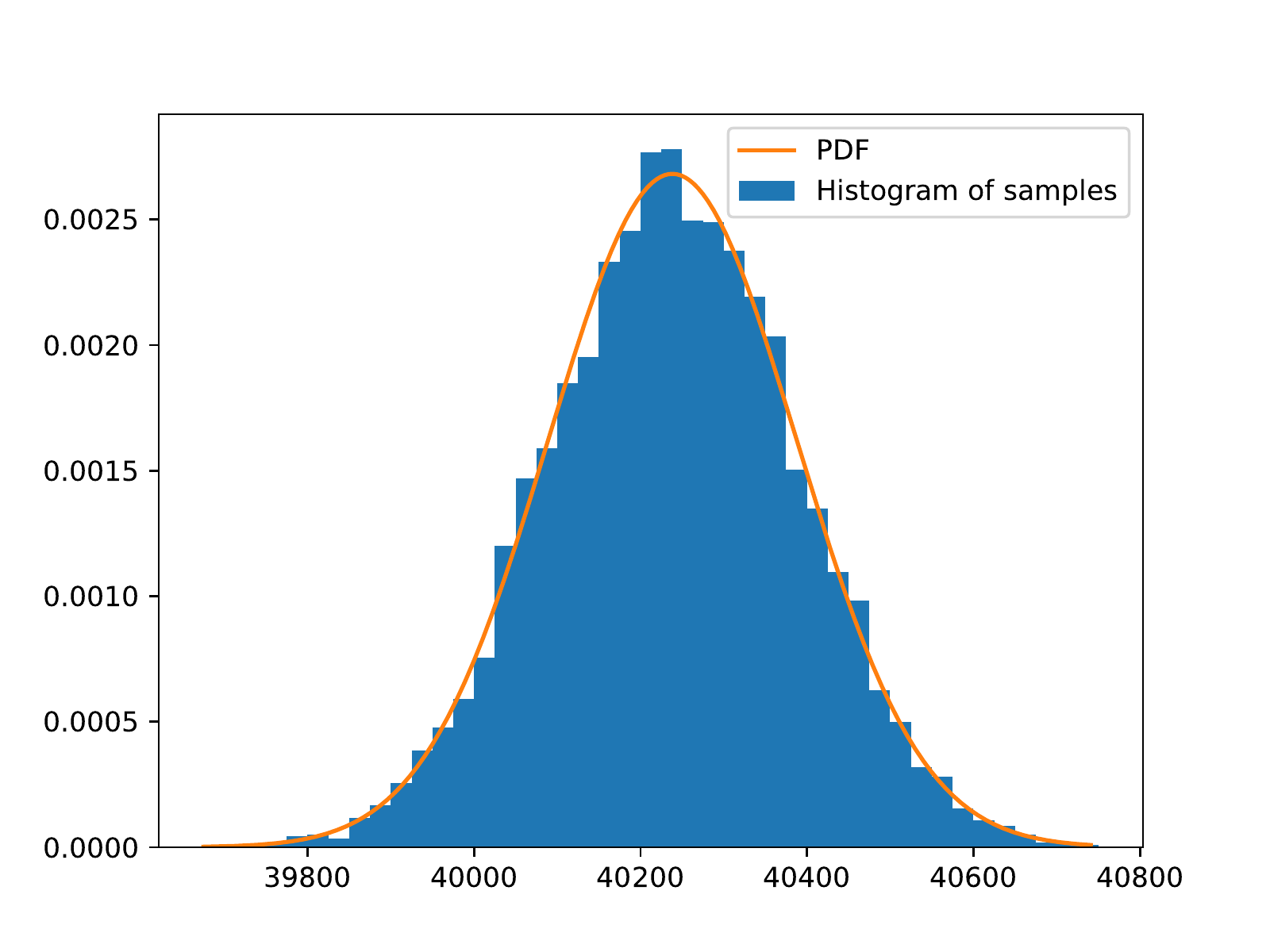}
        \caption{Return strategy}
        \label{pdReturn}
\end{subfigure}%
\begin{subfigure}{.5\textwidth}
  \centering
	    \includegraphics[scale=0.5]{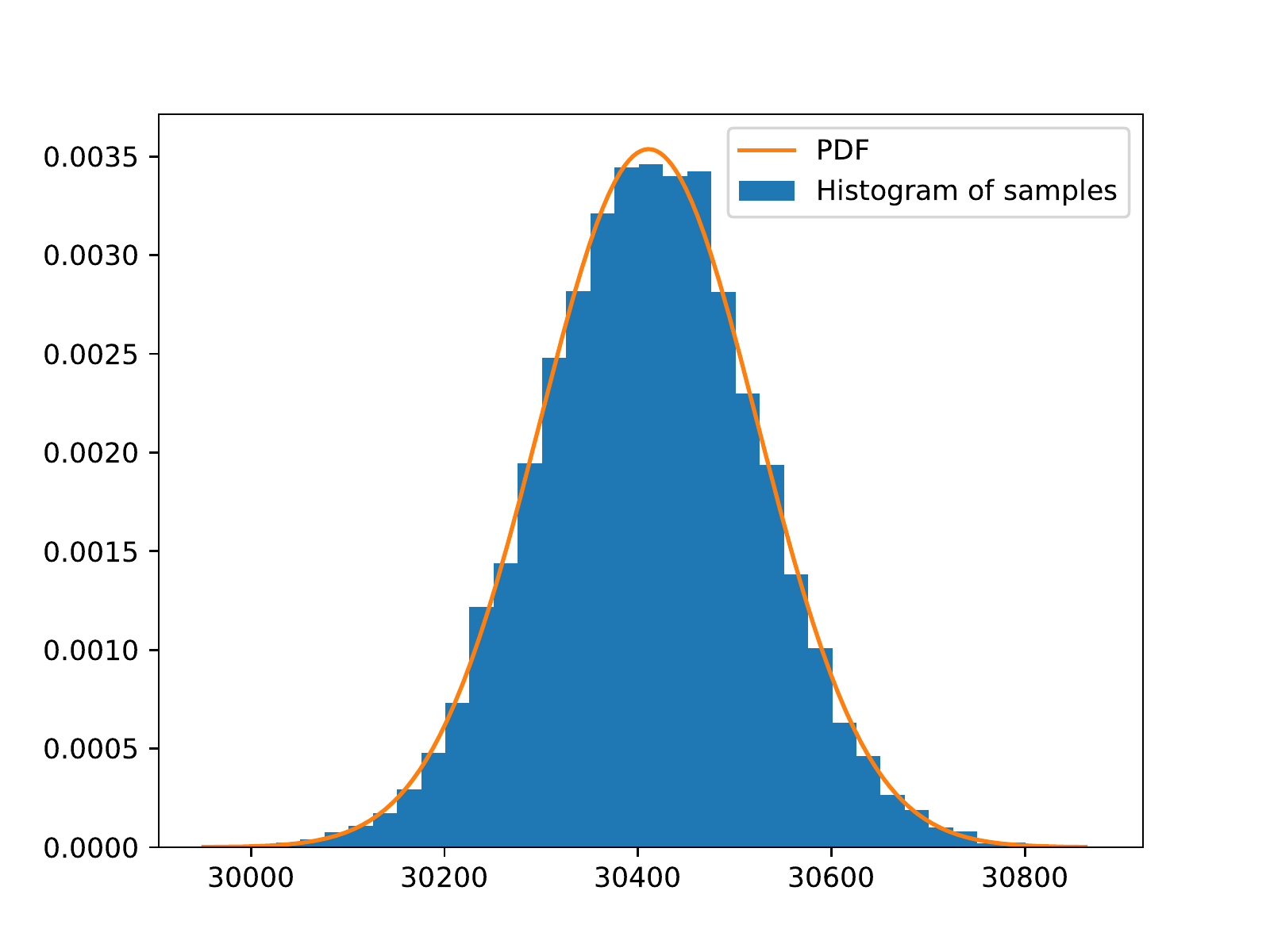}
	    \caption{S-shape strategy}
	    \label{pdSshape}
\end{subfigure}
\caption{Normality of Solution distribution under different routing strategies}
\label{fig:test}
\end{figure}
\subsection{\emph{Correlation-based approximation framework}}

While the distributions of solution values are tested to be normal, how to get a high-quality solution remains unsolved. The main difficulty of directly solving the MIP problem (\ref{eq0}) and (\ref{objs}) is that they are complicated and having are too many variables. Here we give an framework to form an approximation MILP formulation for (\ref{eq0}) and (\ref{objs}). This framework is based on using a random variable $Y^{\pi}$ to approximate $X^R$ and $X^S$ at the same time. We use the correlation coefficient to evaluate the goodness of fit for $X^R$ and $X^S$. The joint distribution functions can provide an upper bound of estimation of $X^R$ and $X^S$. Under such circumstances, if $Y^{\pi}$ represents a simpler and easy-to-solve MILP formulation, then by solving $Y^{\pi}$ we get a $|J|$-partition permutation $p$ which has a high possibility to be a high-quality solution for the formulation (\ref{eq0}) and (\ref{objs}) at the same time.

We define the problem as the Max-Correlation-Reformulation problem(MCRP). For a given random variable $X^{\pi_0}$ with $\pi_0=(f_X,P_X)$, we describe the MCRP as follows.
\begin{equation}
\max Corr(Y^{\pi},X^{\pi_0})-\epsilon ln(\beta),
\label{MCRPobj}
\end{equation}
where $Y^{\pi}$ with $\pi =(f_Y,P_Y)$ and $\beta \in (0,1)$ satisfying
\begin{equation}
\text{dim}(\text{conv}(P_Y))\leq \beta \text{dim}(\text{conv}(P_X)),
\label{MCRPdim}
\end{equation}
and there exists a bijection $g$ such that
\begin{equation}
	g(Proj_y(P_Y))=M_{IJ}.
\label{MCRPproj}
\end{equation}

Here $Corr(Y^{\pi},X^{\pi_0})$ stands for the correlation coefficient for random variables  $Y^{\pi}$ and  $X^{\pi_0}$. The objective function in (\ref{MCRPobj}) aims to seek the maximum correlation coefficient value while ensuring the scale of the formulation represented by $Y^\pi$ is not too large. Generally, a large dimension (i.e. too many variables) would cause a large branching tree, and thus an exponential increment of computing time. Therefore, we set constraint (\ref{MCRPdim}) to limit the dimension of the formulation represented by $Y^\pi$. Constraint (\ref{MCRPproj}) ensures that the formulation represented by $Y^\pi$ does not lose any feasible solution.

We test the normality of $X^R$ and $X^S$ in subsection \ref{sectionnormal}. Since we consider $ Y^{\pi}$ and $X^{\pi_0}$ as random variables, we give the conditional expectation of $X^{\pi_0}$ as below.

\begin{prop}
Random variables $Y^{\pi}$ and $X^{\pi_0}$ are defined by $  p \mapsto   \min \{f(y(p),w): (y(p),w)\in P\}$. If $(X^{\pi_0},Y^{\pi})$ follow bivariate normal distribution
\begin{equation}
\centering
(X^{\pi},Y^{\pi}) \sim N\left(\left(\mu_{X}, \mu_{Y}\right),\left[\begin{array}{cc}
	\sigma_{X}^{2} & \rho \sigma_{X} \sigma_{Y} \\
	\rho \sigma_{X} \sigma_{Y} & \sigma_{Y}^{2}
\end{array}\right]\right)
\label{jointbinormal}
.\end{equation}
Given $Y^\pi=Y_0$, the conditional expectation of $X^{\pi_0}$ is

$$\mathbb{E}[X^{\pi_0} \mid Y^{\pi}=Y_0]=\mu_{X}+\sigma_{X} \rho \frac{Y_0-\mu_{Y}}{\sigma_{Y}}. $$
\end{prop}

The proof can be driven directly from the definition of bivariate normal distribution random variables.

We also provide the confidence interval of $X^{\pi_0}$ given $Y^{\pi}=Y_0$. For interval $(-\infty, z)$, $X^{\pi_0}$ has a probability of $1-\alpha$ to be less than $z$.

\begin{prop}
Random variables $Y^{\pi}$ and $X^{\pi_0}$ are defined by $  p \mapsto   \min \{f(y(p),w): (y(p),w)\in P\}$. If $(X^{\pi},Y^{\pi})$ follow bivariate normal distribution (\ref{jointbinormal}). Given $Y^\pi=Y_0$, the upper bound $z$ that $P(X^{\pi_0} < z| Y^{\pi}=Y_0)=1-\alpha$ is given by

$$ z = \sigma_{X|Y} \Phi ^{-1}(1-\alpha)+\mu_{X|Y},$$

where $\Phi(z)=\frac{1}{\sqrt{2 \pi}} \int_{-\infty}^{z} e^{-t / 2} d t$ and 

$$ \mu_{X|Y}=\mu_{X}+\sigma_{X} \rho \frac{Y_0-\mu_{Y}}{\sigma_{Y}},$$

$$ \sigma_{X|Y}=(1-\rho^{2}) \sigma_{X}^{2}.$$

\end{prop}

\begin{proof}
Since
$$
f_{X \mid Y}\left(X^{\pi_0} \right)=\frac{f_{Y,X}\left(Y^{\pi}, X^{\pi_0}\right)}{f_{Y}\left(Y^{\pi}\right)},
$$
where
$$
\begin{gathered}
	f_{Y, X}\left(Y^{\pi}, X^{\pi_0}\right)= \\
	\frac{1}{2 \pi \sigma_{Y} \sigma_{X} \sqrt{1-\rho}} \exp \left(-\frac{1}{2\left(1-\rho^{2}\right)}\left(\frac{\left(Y^{\pi}-\mu_{Y}\right)^{2}}{\sigma_{Y} ^{2}}+\frac{\left(X^{\pi_0}-\mu_{X}\right)^{2}}{\sigma_{X}^{2}}-\frac{2 \rho\left(Y^{\pi}-\mu_{Y}\right)\left(X^{\pi_0}-\mu_{X}\right)}{\sigma_{Y} \sigma_{X}}\right)\right),
\end{gathered}
$$
and
$$
f_{Y}(Y^{\pi})=\frac{1}{\sqrt{2 \pi \sigma_{Y}^{2}}} \exp \left(-\frac{\left(Y^{\pi}-\mu_{Y}\right)^{2}}{2 \sigma_{Y}^{2}}\right).
$$
The conditional distribution follows
$$
X^{\pi_0} \mid Y^{\pi}=Y_0 \sim N\left(\mu_{X}+\rho \frac{\sigma_{Y}}{\sigma_{X}}\left(Y_0-\mu_{Y}\right),\left(1-\rho^{2}\right) \sigma_{X}^{2}\right)
.$$

The equation

$$P(X^{\pi_0} < z| Y^{\pi}=Y_0)=1-\alpha$$

lead to

$$ \Phi(\frac{z-\mu_{X|Y}}{\sigma_{X|Y}})=1-\alpha.$$

Then we have the result

$$ z = \sigma_{X|Y} \Phi ^{-1}(1-\alpha)+\mu_{X|Y}.$$

\end{proof}

If $(X^{\pi},Y^{\pi})$ does not follow bivariate normal distribution or any known distribution, we can still estimate the upper bound of confidence interval of $X^{\pi_0}$ on condition that $Y^{\pi}=Y_0$ by the empirical distribution function.

\begin{prop}
Random variables $Y^{\pi}$ and $X^{\pi_0}$ are defined by $  p \mapsto   \min \{f(y(p),w): (y(p),w)\in P\}$. Let $\left((X_{1},Y_{1})\ldots, (X_{n},Y_{n})\right)$ be independent, identically distributed random variables with the common cumulative distribution function of $F(u,v)$. Then the empirical distribution function is defined as
$\widehat{F}_{n}(u,v)=\frac{1}{n} \sum_{i=1}^{n} \mathbf{1}_{(X_{i} \leq u, Y_{i}\leq v)}$

$\widehat{F}_{n}(v)=\frac{1}{n} \sum_{i=1}^{n} \mathbf{1}_{Y_{i}\leq v}$
where $\mathbf{1}_{A}$ is the indicator of event $A$. Given $Y^{\pi} \leq Y_0$ and a real number $z$,

$$ P(X^{\pi_0} \leq z | Y^{\pi} \leq Y_0)= \frac{\widehat{F}_{n}(z,Y_0)}{\widehat{F}_{n}(Y_0)} =\frac{\sum_{i=1}^{n} \mathbf{1}_{(X_{i} \leq z, Y_{i}\leq Y_0)}}{\sum_{i=1}^{n} \mathbf{1}_{Y_{i}\leq Y_0}}.$$
\end{prop}

The proof can be driven directly from the definition of the cumulative distribution function.

\section{Construct Solution for the MCRP}
\label{sol}
We present a two-phase method for solving the MCRP heuristically in this section. Given $X^{\pi_0}$ equals to $X^R$ or $X^S$, the plan is to acquire a good correlation coefficient value and meanwhile a small dimension ratio $\beta$ for $Y^{\pi}$. In the first phase, we find a simplified model (i.e. polytope a with lower dimension) from the past classic heuristics. In the second phase, with the fixed polytope, we try to find a better objective function with a high correlation coefficient with the original model.

\subsection{\emph{Decrease the dimension of formulation}}
\label{sol1}
We find our initial approximation model using examples of past heuristics because one reason for their success is that they capture vital features of the problem. By setting rules based on these features they find high-quality solutions. Therefore, by experiment, we find that an approximation formulation that minimizes the number of aisles traversed in each batch gives a high correlation coefficient with the exact one. It is a variant of the Seed heuristic \citep{de1999efficient}, which shows a link between the reduction of reduction in the number of aisles visited in each batch and the decrease in both total travel distance and the number of unique items.

 Since the approximation model only considers whether an aisle is visited in a batch, the other information provided by each order is redundant. Consider the orders that include items from the same aisles as a pattern $t$. The total number of patterns must be less than the total number of orders. Therefore, we substitute $y_{ij}$ with the variable $x_{tj}$ to decide how many orders of pattern $t$ are assigned to batch $j$. This replacement does not exclude feasible solutions, but it does lower the model's scale. 

Denote $Q_t$ the set of custom orders of pattern t. Let	$T$ be the set of all patterns and $t$ its index $t \in T$. The set of patterns that visit aisle $b$ is denoted by $T_b$. The variable $x_{tj}$ determines how many orders of pattern $t$ are assigned to batch $j$. As a result, the final approximation model is stated as follows. 

\begin{equation}
	\min \{ \sum_{j \in J} \sum_{b \in B}  \delta_{jb}:(x_{tj},\delta_{jb}) \in P^A _{IJB}\}
\end{equation}

where $P^A _{IJB}$ is given by,
\begin{align}
	\sum_{j \in J} x_{tj}& = |Q_t| , \quad \forall t \in T;
	\label{peq0}\\
	\sum_{t \in T} x_{tj} &= C, \quad \forall j \in J;
	\label{peq1}\\
	M_b \delta_{jb} &\geq \sum_{t \in T_b} x_{tj}, \quad j \in J, \quad b \in B;\label{peq2}\\
	x_{tj} &\in \{0,1,...,|Q_t|\},\quad \forall t \in T,\quad \forall j \in J;\label{peq3}\\
	\delta_{jb} &\in \{0,1\},\quad \forall j \in J, \quad \forall b \in B;\label{peq4}
\end{align}
Constraint (\ref{peq0}) is equivalent to equation (\ref{eq1}), which ensures each order is assigned to a batch. Constraint (\ref{peq1}) prevents that the number of orders in a batch exceeds the packing point's capacity. Constraint (\ref{peq2}) determines whether aisle $b$ is visited in batch $j$ with a big-M inequality. Constraints (\ref{peq3}) and (\ref{peq4}) give domains of variables $x_{tj}$ and $\delta_{jb}$

Denote $\pi=(\sum _{j \in J} \sum_{b \in B} \delta_{jb}, \ P^A _{IJB})$. To investigate the correlation coefficient between $Y^\pi$ and $X^R$(or $X^S)$, we conduct the same Monte Carlo Sampling process as described above with n=5000. The statistic are recorded as $Y^{\pi}_n$, $X^R _n$ and $X^S _n$.

The experiment results are shown below. We plot scatter figures for $Y^{\pi}_n$ and $X^R _n$ (See Figure \ref{cpdReturn}), $Y^{\pi}_n$ and $X^S _n$ (See Figure \ref{cpdSshape}). The correlation coefficient values are 0.61 and 0.85 corresponding to two strategies. It is also tested that ($Y^{\pi}$, $X^R$) and ($Y^{\pi}$, $X^S$) follow bivariate normal distributions. We plot confidence ellipses for them with $\sigma=1$ and standard deviations equal to $1 \sigma$, $2\sigma$ and $3\sigma$.  

\begin{figure}
\centering
\begin{subfigure}{.5\textwidth}
  \centering
  \includegraphics[scale=0.25]{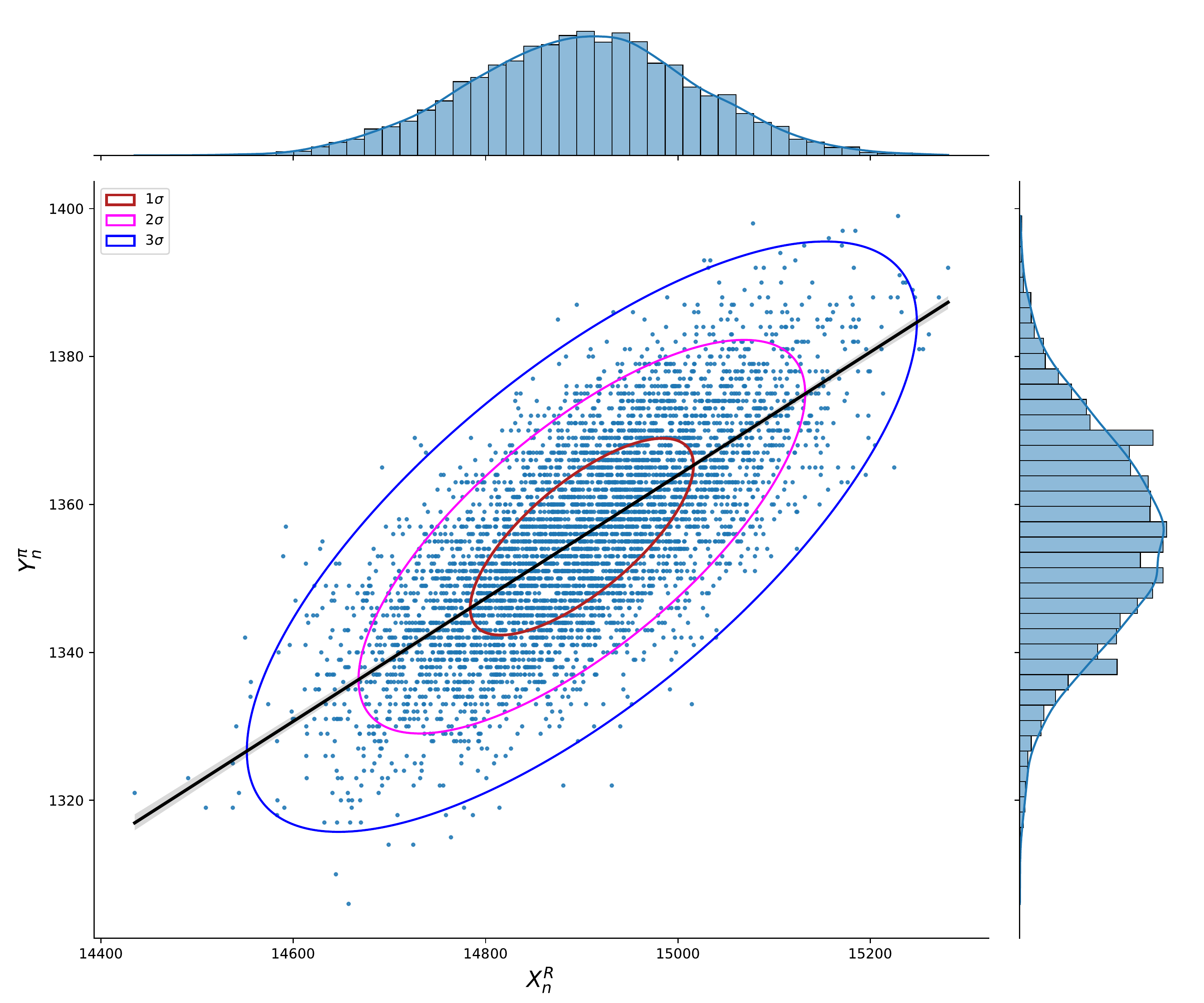}
  \caption{Return strategy with Corr=0.71}
  \label{cpdReturn}
\end{subfigure}%
\begin{subfigure}{.5\textwidth}
  \centering
	\includegraphics[scale=0.25]{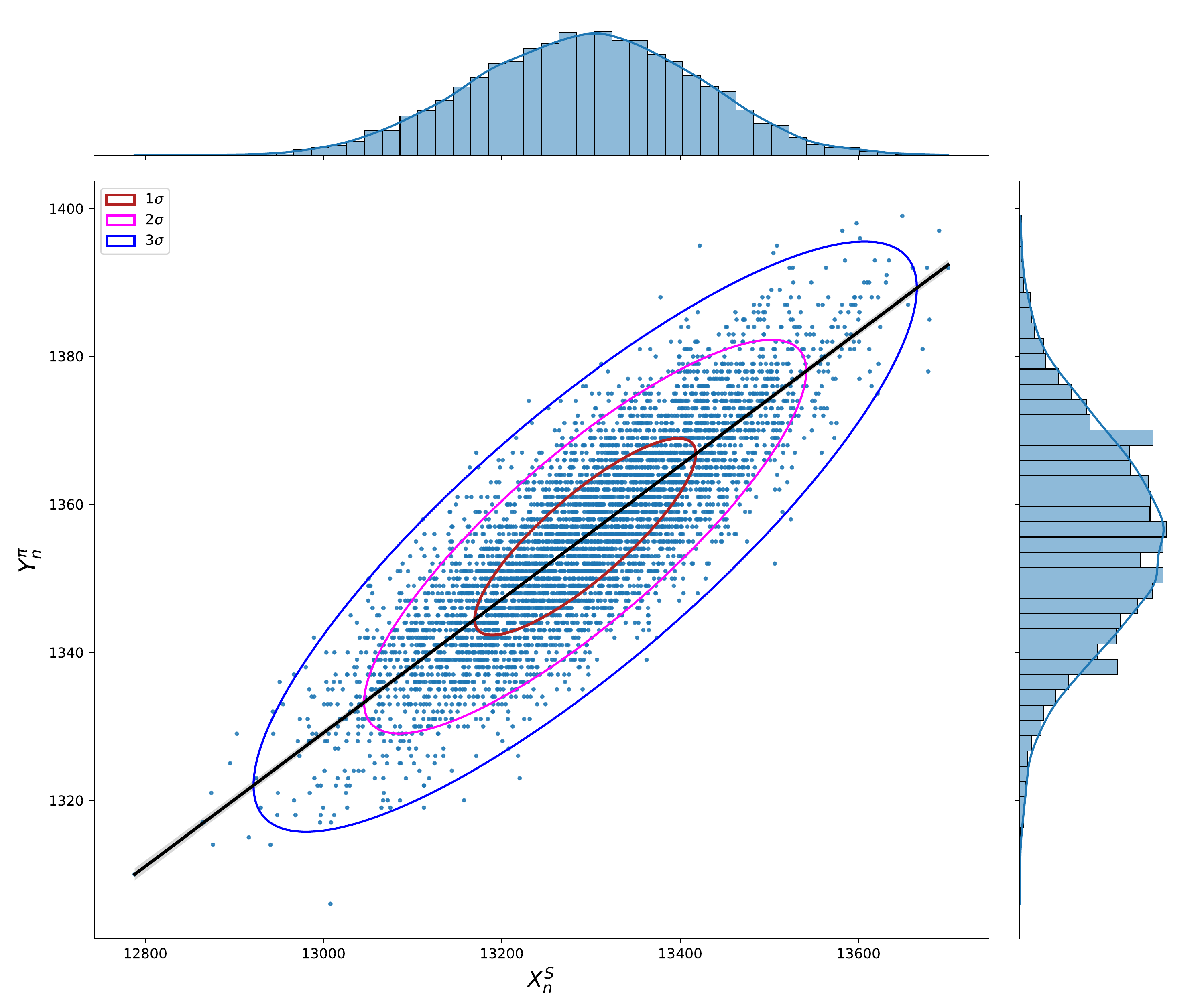}
	\caption{S-shape strategy with Corr=0.84 }
	\label{cpdSshape}
\end{subfigure}
\caption{Solution distribution with $\pi=(\sum _{j \in J} \sum_{b \in B} \delta_{jb}, \ P^A _{IJB})$}
\label{fig:test1}
\end{figure}

\subsection{{Increase the correlation coefficient}}

\label{sol2}

In the second phase, we increase the correlation coefficient by changing objective function with a fixed formulation.

To strengthen the approximation of the model, we consider the length of the aisles. As shown in the parameter table, there are $|B|$ aisles in the warehouse. The lengths of these aisles decrease with their serial numbers. We set the lengths of aisles as weights for the variable $\delta_{jb}$. Hence, we denote

$$\pi_{new}=\left(\sum _{j \in J} \sum_{b \in B} l_b \delta_{jb}, \ P^A _{IJB}\right).$$

To show the improvement, we conduct the Monte Carlo experiment with n=5000 and compute the correlation coefficient values. The scatter figures for $Y^{\pi}_n$ and $X^R _n$ (See Figure \ref{cnpdReturn}) and $X^S _n$ (See Figure \ref{cnpdSshape}) with new objective function are plotted below. The correlation coefficient values increase to 0.72 and 0.97 corresponding to the two strategies. The random variables ($Y^{\pi}$, $X^R$) and ($Y^{\pi}$, $X^S$) still follow bivariate normal distributions.

\begin{figure}
\centering
\begin{subfigure}{.5\textwidth}
  \centering
  \includegraphics[scale=0.25]{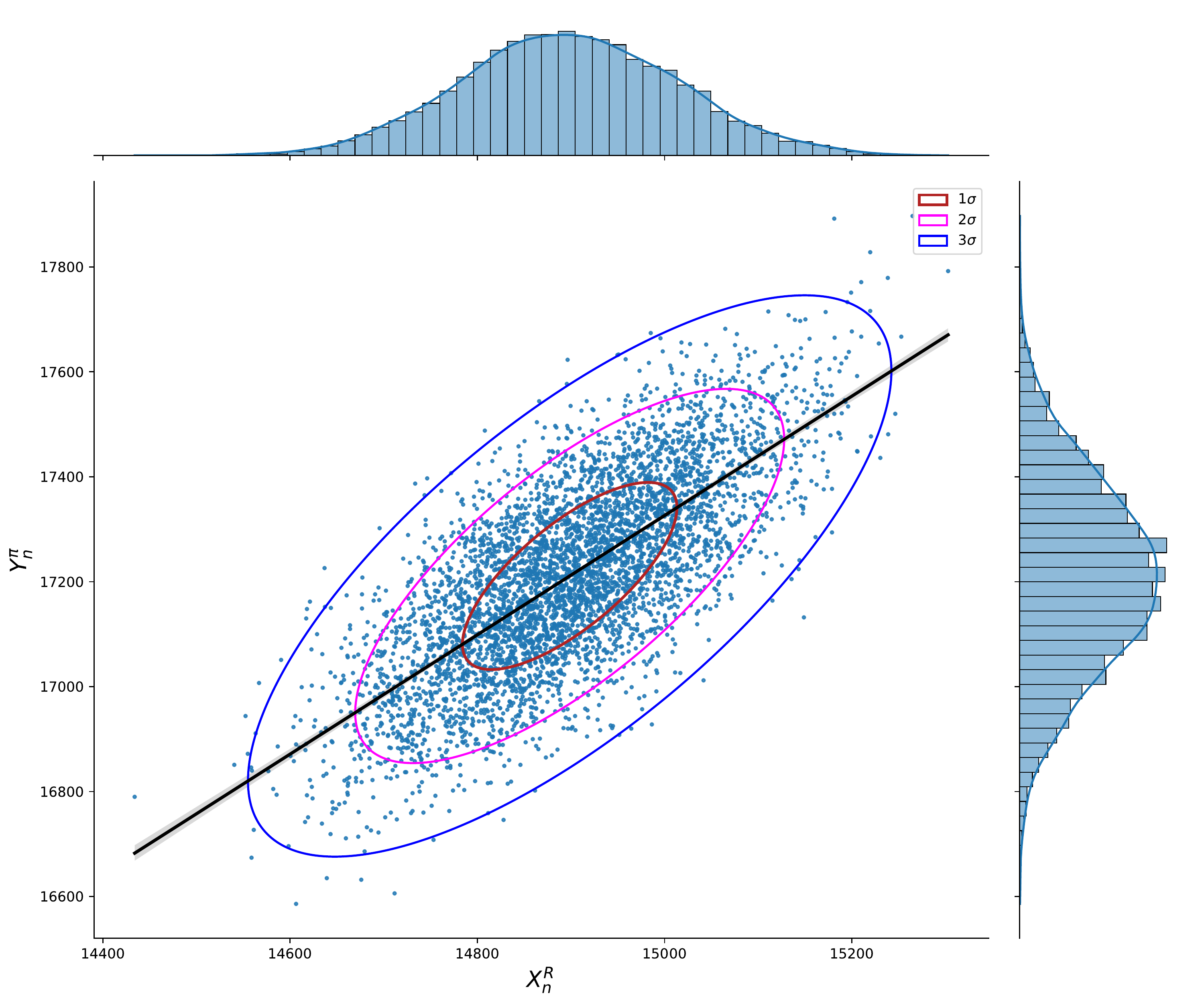}
  \caption{Return strategy with Corr=0.72}
  \label{cnpdReturn}
\end{subfigure}%
\begin{subfigure}{.5\textwidth}
  \centering
  \includegraphics[scale=0.25]{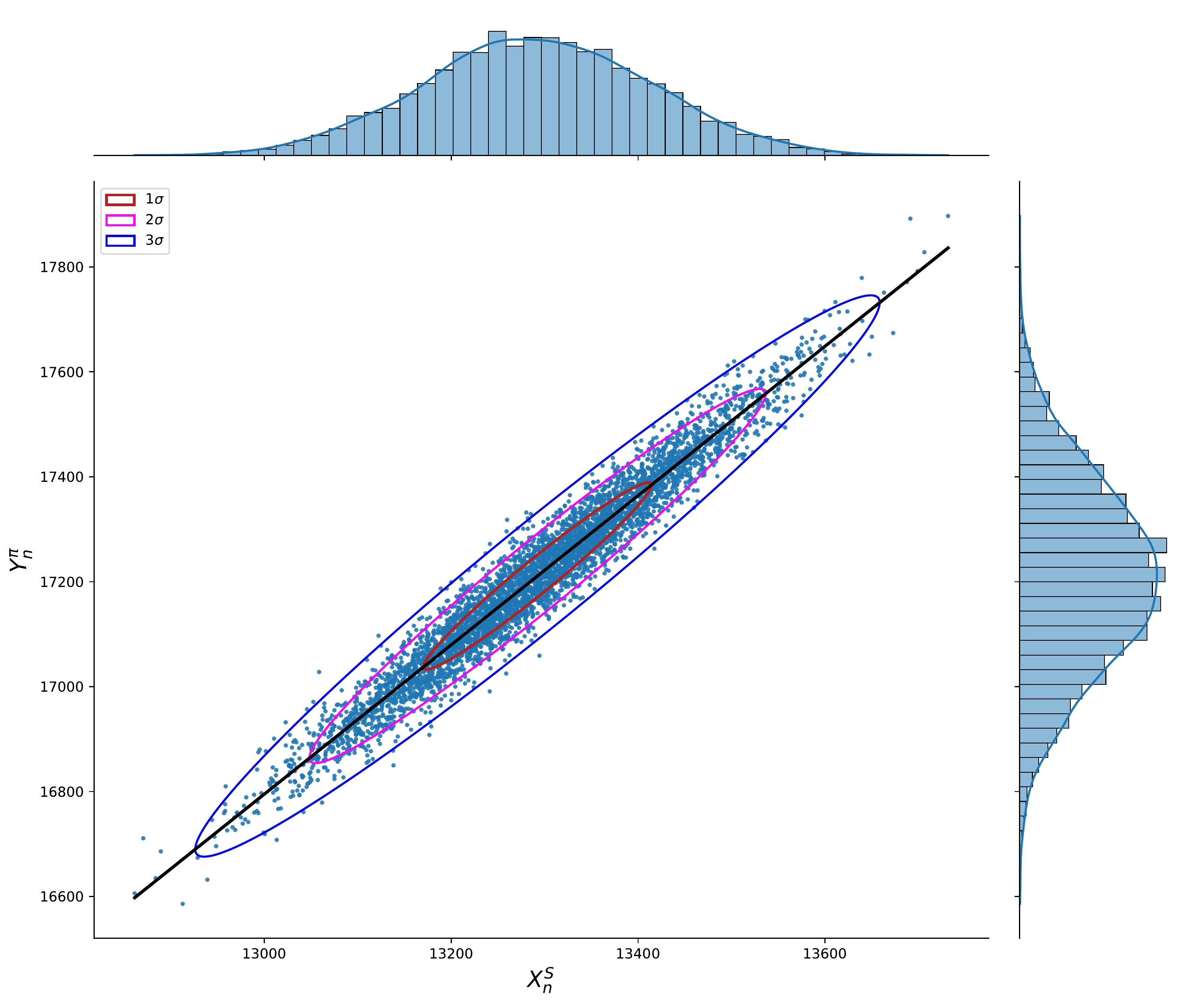}
  \caption{S-shape strategy with Corr=0.97}
  \label{cnpdSshape}
\end{subfigure}
\caption{Solution distribution with $\pi_{new}=(\sum _{j \in J} \sum_{b \in B} l_b \delta_{jb}, \ P^A _{IJB})$}
\label{fig:test2}
\end{figure}

\section{Integrated Heuristic}
\label{int}

The approximation formulation in subsection \ref{sol2} can be solved by a MILP solver. To further improve our solution for OBOPP, we consider constructing a local search-like heuristic as an integrated method.

A vital aspect of designing a local search-type algorithm for OBOPP is to find a well-defined neighborhood. Because OBOPP is high-dimensional, the search space is vast. Therefore, if we naively define the neighborhood (e.g. arbitrarily swap two orders from separate batches), we might cause an inefficient searching procedure with a significant increase in computation time.

The previous section has shown that the number of visited aisles, as an objective function, can still maintain part of the solution structure. It indicates that we can select orders with similar aisle distributions to avoid increasing the number of visited aisles but yet lowering the real objective value by swapping them. Therefore, we consider $k$-different pairs. Denote $A_1$ and $A_2$ the set of aisles visited in orders $i_1$ and $i_2$. A pair of order $(i_1,i_2)$ is called $k$-different pairs if $|A_1\backslash A_2|\leq k$ and $|A_2\backslash  A_1| \leq k$. Finally, we also set the minimum order size $m$ to ensure that we only generate the order pairs that have the potential to cause a significant decrease in objective value.

We describe the algorithm as the valuable pairs generation (VPG) method:
\begin{algorithm}[H]
	\caption{Valuable Pairs Generation}
	\begin{algorithmic}[1]  
		\REQUIRE order sequence $i_1$, $i_2$,...,$i_n$, empty pair list $L$
		\FOR {$p = 1: n-1$}
		\FOR {$q = p+1:n$}
		\IF{$|A_p \backslash A_q|\leq k$, $|A_q\backslash A_p| \leq k$ and orders $i_p$, $i_q$ have more than $m$ items}
		\STATE {Add $(i_p,i_q)$ to pair list $L$}
		\ENDIF
		\ENDFOR
		\ENDFOR
	\end{algorithmic} 
\label{ag}
\end{algorithm}
 We apply the pair list to the solution maintained by \ref{sol2}. In this process, we scan the list and swap all pairs in separate batches which can decrease the objective value. We set $k=2$ and $m=1$ to achieve a satisfactory trade-off between the full strength and efficiency of this algorithm in practice.

\section{Experiment results}
\label{exp}

\subsection{\emph{Experiment design}}
To reveal the strength of our methods (i.e., AP1: the approximation MILP proposed in subsection \ref{sol1} AP2: the approximation MILP proposed in subsection \ref{sol2}, and AP2+VPG: the VPG algorithm with AP2 providing an initial solution), we also choose the following Exact approach and classic and frequently-used heuristics in OBPs:

\begin{tabular}{rl}
	
Exact&: solve formulation (\ref{eq0}) and (\ref{objs}) with an MIP solver. \\
ILS&: iterated local search in \cite{henn2010metaheuristics}. \\
Seed&: the seed algorithm in \cite{de1999efficient}. \\
CWI&: the Clarke and Wright algorithm (I) in \cite{de1999efficient}\\
CWII&: the Clarke and Wright algorithm (II) in \cite{de1999efficient}
\end{tabular}

The datasets were acquired from a warehouse in China. The datasets contain information about the warehouse's layout. The warehouse contains 541 storage locations for 6413 unique items. A storage location's length is one length unit (LU). Aisles range in length from 1 to 36 LU and are placed in decreasing order from left to right. All the aisles align along the bottom line of the warehouse. The width of the warehouse is 48 LU. The depot is located on the left side of the aisles, while the packing point is located on the right side of the aisles. The factor $\alpha$ is set to be 0.4.

There are 800 orders in both datasets 1 and 2. As mentioned in the previous sections, we use dataset 1 as a statistical object. In this experiment, we use dataset 2 as real-world test data (Test(800)) to examine our method. Each order has an average of 2.8 items in it. The number of items to unique items ratio is approximately 3:2. All other test data are randomly generated with the same parameters. First, select a set of unique items arbitrarily from all item types. Assuming that all of the items in the order are empty slots, divide them by n numbers from the quantity of the slot, where n is the number of orders. Finally, randomly place the unique items in the slots.
\begin{table}
\caption{Experiment profiles and parameters}
\begin{tabular}{ll}
	\hline Profiles & Notations and Parameters \\
	\hline Algorithms & Exact, AP1, AP2, AP2+VPG, ILS, Seed, CWI, CW II \\
 	Routing policy & Return strategy, S-shape strategy\\
	The number of orders(N) & $480, 600, 720, \text{Test}(800), 840,960,1080,1200$ \\
	Packing capacity (C)& $10,15, \text{Test}(16), 20$ \\
	\hline
\end{tabular}
\end{table}

We use the Python 3.8 language and the academic version of GUROBI 9.5.1 for MIP solving. The programs run on a PC (Intel i7 2.60 gigahertz CPU, 16-gigabyte memory) with Windows 10.0. The time limit for solving exact models is set to 1000 seconds. The time limit to solve our approximate model is set to 100 seconds. Each experiment is formed of 20 random instances. 

To evaluate experiment results, we denote Obj, LB, and Gap as the objective value, the linear relaxation value, and the integrality gap of the exact models. Let CPU be the running time of algorithms. Finally, we use the Obj ratio to compare the solution quality of two algorithms, which is the ratio of their objective values.

\subsection{\emph{Experiment results}}

\begin{sidewaystable}
  \centering
  \caption{Comparison with Exact models}
 \begin{adjustbox}{scale=0.7,center}
    \begin{tabular}{crrrrrrrrrrrrrrr}
    \toprule
    \multirow{3}[0]{*}{N} & \multicolumn{1}{c}{\multirow{3}[0]{*}{C}} & \multicolumn{7}{c}{Return strategy}                   & \multicolumn{7}{c}{S-shape strategy} \\
    \cmidrule(lr){3-9} 
    \cmidrule(lr){10-16}
          &       & \multicolumn{3}{c}{Exact} & \multicolumn{1}{c}{AP1} & \multicolumn{1}{c}{AP2} & \multicolumn{2}{c}{Obj ratio} & \multicolumn{3}{c}{Exact} & \multicolumn{1}{c}{AP1} & \multicolumn{1}{c}{AP2} & \multicolumn{2}{c}{Obj ratio} \\
          \cmidrule(lr){3-5}
          \cmidrule(lr){8-9}
          \cmidrule(lr){10-12}
          \cmidrule(lr){15-16}
          &       & \multicolumn{1}{l}{Obj} & \multicolumn{1}{l}{Gap (\%)} & \multicolumn{1}{l}{LB} & \multicolumn{1}{l}{Obj} & \multicolumn{1}{l}{Obj} & \multicolumn{1}{l}{Exact/AP2} & \multicolumn{1}{l}{AP1/AP2} & \multicolumn{1}{l}{Obj} & \multicolumn{1}{l}{Gap (\%)} & \multicolumn{1}{l}{LB} & \multicolumn{1}{l}{Obj} & \multicolumn{1}{l}{Obj} & \multicolumn{1}{l}{Exact/AP2} & \multicolumn{1}{l}{AP1/AP2} \\
          \hline
    \multirow{3}[0]{*}{480} & 10    & 8627.28  & 73.56  & 2280.88  & 6698.80  & 6718.88  & 1.28  & 1.00  & 13165.76  & 60.04  & 5261.44  & 5593.34  & 5403.74  & 2.44  & 1.04  \\
          & 15    & 7509.28  & 73.93  & 1957.80  & 6133.50  & 6140.00  & 1.22  & 1.00  & 10188.68  & 55.86  & 4496.80  & 5040.24  & 4861.90  & 2.10  & 1.04  \\
          & 20    & 7543.44  & 73.84  & 1973.01  & 5707.43  & 5720.44  & 1.32  & 1.00  & 10447.95  & 55.69  & 4629.49  & 4625.81  & 4460.47  & 2.34  & 1.04  \\
    \cmidrule(lr){1-2}
    \multirow{3}[0]{*}{600} & 10    & 8450.32  & 75.77  & 2047.64  & 6352.47  & 6350.22  & 1.33  & 1.00  & 10677.38  & 58.60  & 4420.32  & 5195.88  & 5023.38  & 2.13  & 1.03  \\
          & 15    & 9088.35  & 75.54  & 2222.96  & 6464.94  & 6474.41  & 1.40  & 1.00  & 12349.76  & 60.93  & 4824.62  & 5265.70  & 5102.97  & 2.42  & 1.03  \\
          & 20    & 8413.19  & 76.11  & 2009.59  & 6386.47  & 6403.47  & 1.31  & 1.00  & 10260.32  & 59.13  & 4193.13  & 5165.47  & 5007.80  & 2.05  & 1.03  \\
    \cmidrule(lr){1-2}
    \multirow{3}[0]{*}{720} & 10    & 10036.55  & 76.16  & 2392.56  & 6882.47  & 6899.14  & 1.45  & 1.00  & 14414.79  & 65.54  & 4967.85  & 5596.80  & 5428.75  & 2.66  & 1.03  \\
          & 15    & 9332.99  & 76.39  & 2203.58  & 7076.95  & 7100.19  & 1.31  & 1.00  & 12057.93  & 63.07  & 4453.23  & 5729.05  & 5558.25  & 2.17  & 1.03  \\
          & 20    & 10056.57  & 76.12  & 2401.33  & 7084.74  & 7113.50  & 1.41  & 1.00  & 14298.96  & 65.39  & 4949.22  & 5710.89  & 5532.87  & 2.58  & 1.03  \\
    \cmidrule(lr){1-2}
    Test(800) & 16    & 12645.60  & 79.96  & 2533.60  & 8993.59  & 9028.40  & 1.40  & 1.00  & 14910.00  & 64.23  & 5333.20  & 6982.20  & 6628.00  & 2.25  & 1.05  \\
    \cmidrule(lr){1-2}
    \multirow{3}[0]{*}{840} & 10    & 10105.54  & 76.77  & 2347.96  & 7519.19  & 7552.84  & 1.34  & 1.00  & 13542.12  & 66.53  & 4532.23  & 6109.54  & 5915.38  & 2.29  & 1.03  \\
          & 15    & 11136.15  & 76.48  & 2619.35  & 7709.84  & 7745.83  & 1.44  & 1.00  & 16934.76  & 71.99  & 4742.84  & 6257.17  & 6061.88  & 2.79  & 1.03  \\
          & 20    & 10424.43  & 77.26  & 2370.35  & 7769.60  & 7803.75  & 1.34  & 1.00  & 13571.38  & 67.88  & 4358.45  & 6278.65  & 6077.48  & 2.23  & 1.03  \\
    \cmidrule(lr){1-2}
    \multirow{3}[0]{*}{960} & 10    & 12037.20  & 76.79  & 2794.38  & 8178.11  & 8201.45  & 1.47  & 1.00  & 19122.78  & 74.42  & 4891.09  & 6655.76  & 6423.42  & 2.98  & 1.04  \\
          & 15    & 11360.53  & 77.55  & 2550.66  & 8379.37  & 8396.32  & 1.35  & 1.00  & 15448.19  & 70.95  & 4487.21  & 6814.77  & 6574.03  & 2.35  & 1.04  \\
          & 20    & 12319.39  & 77.02  & 2830.50  & 8457.79  & 8478.57  & 1.45  & 1.00  & 19046.82  & 75.87  & 4596.53  & 6854.38  & 6617.63  & 2.88  & 1.04  \\
    \cmidrule(lr){1-2}
    \multirow{3}[0]{*}{1080} & 10    & 12360.65  & 78.26  & 2687.01  & 8858.92  & 8864.62  & 1.39  & 1.00  & 17412.96  & 73.43  & 4627.48  & 7215.58  & 6952.85  & 2.50  & 1.04  \\
          & 15    & 13715.08  & 77.86  & 3036.04  & 9059.52  & 9067.44  & 1.51  & 1.00  & 21523.68  & 77.21  & 4906.30  & 7372.41  & 7111.42  & 3.03  & 1.04  \\
          & 20    & 12946.40  & 78.91  & 2730.40  & 9153.62  & 9158.20  & 1.41  & 1.00  & 17756.36  & 72.25  & 4927.42  & 7427.41  & 7167.10  & 2.48  & 1.04  \\
    \cmidrule(lr){1-2}
    \multirow{3}[0]{*}{1200} & 10    & 14465.64  & 77.94  & 3190.91  & 9530.86  & 9519.02  & 1.52  & 1.00  & 23560.75  & 80.22  & 4661.13  & 7765.64  & 7483.53  & 3.15  & 1.04  \\
          & 15    & 13501.53  & 78.68  & 2878.80  & 9733.12  & 9727.31  & 1.39  & 1.00  & 19415.99  & 75.37  & 4781.56  & 7931.61  & 7645.97  & 2.54  & 1.04  \\
          & 20    & 15284.89  & 80.22  & 3022.92  & 9839.38  & 9832.71  & 1.55  & 1.00  & 23524.38  & 80.07  & 4687.50  & 8000.66  & 7710.64  & 3.05  & 1.04  \\
    \bottomrule
    \end{tabular}%
    \end{adjustbox}
  \label{ComERS}%
\end{sidewaystable}

\subsubsection{\emph{Comparison with Exact models}}
We first compare our simplified model with the exact models to examine its improvements. Table \ref{ComERS} shows the results. It indicates that under the return strategy AP2 roughly gives a 20-55\% reduction of the objective value compared with exact models, while the solver used 1000 seconds for the exact models and only 100 seconds for our approximation models. The difference between the two approximation models is not significant. This corresponds to the very little improvement of the correlation coefficient between them. Under the S-shape strategy, AP2 offers a roughly 110-215\% reduction of objective value. This is mainly because the original model for S-shape has a larger scale and is more difficult to solve. The difference between AP1 and AP2 is also very significant, corresponding to the improvement of the correlation coefficient, which increases from 0.84 to 0.97.

\subsubsection{\emph{Comparison with four commonly used algorithms}}
\begin{sidewaystable}
\caption{Return strategy}
\label{Com_R}
\centering
\begin{adjustbox}{scale=0.75,center}
        \begin{tabular}{crrrrrrrrrrrrrrr}
        \toprule
          &       & \multicolumn{2}{c}{AP2+VPG} & \multicolumn{3}{c}{ILS} & \multicolumn{3}{c}{Seed} & \multicolumn{3}{c}{CWI} & \multicolumn{3}{c}{CWII} \\
          \cmidrule(lr){3-4}
          \cmidrule(lr){5-7}
          \cmidrule(lr){8-10}
          \cmidrule(lr){11-13}
          \cmidrule(lr){14-16}
    \multicolumn{1}{l}{N} & \multicolumn{1}{l}{C} & \multicolumn{1}{l}{Obj} & \multicolumn{1}{l}{CPU} & \multicolumn{1}{l}{Obj} & \multicolumn{1}{l}{CPU} & \multicolumn{1}{l}{Obj ratio} & \multicolumn{1}{l}{Obj} & \multicolumn{1}{l}{CPU} & \multicolumn{1}{l}{Obj ratio} & \multicolumn{1}{l}{Obj} & \multicolumn{1}{l}{CPU} & \multicolumn{1}{l}{Obj ratio} & \multicolumn{1}{l}{Obj} & \multicolumn{1}{l}{CPU} & \multicolumn{1}{l}{Obj ratio} \\
    \hline
    \multirow{3}[0]{*}{480 } & 10    & \textbf{6437.52 } & 109.66  & 6755.60  & 9.62  & 1.05  & 9200.44  & 12.74  & 1.43  & 7677.40  & 14.43  & 1.19  & 6900.24  & 35.20  & \textbf{1.07}  \\
    
          & 15    & \textbf{5952.00 } & 109.71  & 6177.32  & 10.30  & 1.04  & 8445.00  & 13.52  & 1.42  & 7110.16  & 14.32  & 1.19  & 6358.94  & 37.05  & \textbf{1.07}  \\
          & 20    & \textbf{5563.40 } & 110.07  & 5767.15  & 10.89  & 1.04  & 7894.21  & 14.13  & 1.42  & 6682.03  & 14.30  & 1.20  & 5957.60  & 38.51  & \textbf{1.07}  \\
        \cmidrule(lr){1-2}
    \multirow{3}[0]{*}{600 } & 10    & \textbf{6176.00 } & 112.08  & 6432.12  & 10.90  & 1.04  & 8742.58  & 15.65  & 1.42  & 7395.72  & 17.15  & 1.20  & 6588.51  & 45.45  & \textbf{1.07}  \\
          & 15    & \textbf{6287.97 } & 113.31  & 6547.43  & 11.49  & 1.04  & 8928.73  & 16.99  & 1.42  & 7569.80  & 18.62  & 1.20  & 6707.65  & 50.84  & \textbf{1.07}  \\
          & 20    & \textbf{6209.72 } & 114.31  & 6478.05  & 12.12  & 1.04  & 8833.69  & 18.19  & 1.42  & 7518.51  & 19.52  & 1.21  & 6649.53  & 55.31  & \textbf{1.07}  \\
        \cmidrule(lr){1-2}
    \multirow{3}[0]{*}{720 } & 10    & \textbf{6679.29 } & 116.33  & 7012.89  & 12.01  & 1.05  & 9501.10  & 19.77  & 1.42  & 8039.66  & 22.20  & 1.20  & 7098.39  & 62.87  & \textbf{1.06}  \\
          & 15    & \textbf{6843.01 } & 117.92  & 7206.56  & 12.24  & 1.05  & 9743.55  & 21.36  & 1.42  & 8268.93  & 24.29  & 1.21  & 7281.77  & 70.43  & \textbf{1.06}  \\
          & 20    & \textbf{6846.67 } & 119.35  & 7233.20  & 12.70  & 1.06  & 9767.19  & 22.88  & 1.43  & 8329.98  & 25.84  & 1.22  & 7314.01  & 77.12  & \textbf{1.07}  \\
        \cmidrule(lr){1-2}
    \multicolumn{1}{l}{Test(800)} & 16    & \textbf{8490.80 } & 142.76  & 9358.40  & 16.10  & 1.10  & 12654.40  & 41.48  & 1.49  & 10832.40  & 52.95  & 1.28  & 9305.20  & 170.48  & \textbf{1.10}  \\
        \cmidrule(lr){1-2}
    \multirow{3}[0]{*}{840 } & 10    & \textbf{7263.26 } & 121.71  & 7706.70  & 12.69  & 1.06  & 10354.30  & 24.57  & 1.43  & 8831.04  & 29.55  & 1.22  & 7751.76  & 87.18  & \textbf{1.07}  \\
          & 15    & \textbf{7430.19 } & 123.87  & 7914.28  & 12.90  & 1.07  & 10613.47  & 26.38  & 1.43  & 9065.38  & 32.19  & 1.22  & 7940.73  & 96.45  & \textbf{1.07}  \\
          & 20    & \textbf{7473.27 } & 125.71  & 7975.38  & 13.39  & 1.07  & 10691.87  & 28.21  & 1.43  & 9157.84  & 34.21  & 1.23  & 8002.13  & 105.39  & \textbf{1.07}  \\
        \cmidrule(lr){1-2}
    \multirow{3}[0]{*}{960 } & 10    & \textbf{7853.34 } & 128.60  & 8419.99  & 13.36  & 1.07  & 11226.44  & 30.05  & 1.43  & 9617.43  & 38.29  & 1.22  & 8404.25  & 117.25  & \textbf{1.07}  \\
          & 15    & \textbf{8027.90 } & 131.09  & 8633.87  & 13.56  & 1.08  & 11489.27  & 32.07  & 1.43  & 9859.35  & 41.46  & 1.23  & 8595.88  & 128.88  & \textbf{1.07}  \\
          & 20    & \textbf{8098.06 } & 133.40  & 8727.13  & 13.96  & 1.08  & 11593.41  & 34.13  & 1.43  & 9978.17  & 44.04  & 1.23  & 8677.21  & 140.23  & \textbf{1.07}  \\
        \cmidrule(lr){1-2}
    \multirow{3}[0]{*}{1080 } & 10    & \textbf{8459.15 } & 136.97  & 9152.01  & 14.01  & 1.08  & 12107.41  & 36.18  & 1.43  & 10426.12  & 48.69  & 1.23  & 9061.70  & 154.15  & \textbf{1.07}  \\
          & 15    & \textbf{8637.14 } & 140.02  & 9378.53  & 14.23  & 1.09  & 12386.76  & 38.47  & 1.43  & 10681.19  & 52.88  & 1.24  & 9260.47  & 168.47  & \textbf{1.07}  \\
          & 20    & \textbf{8718.08 } & 142.96  & 9489.61  & 14.52  & 1.09  & 12515.44  & 40.92  & 1.44  & 10812.06  & 56.33  & 1.24  & 9357.72  & 183.10  & \textbf{1.07}  \\
        \cmidrule(lr){1-2}
    \multirow{3}[0]{*}{1200 } & 10    & \textbf{9067.02 } & 146.87  & 9897.24  & 14.62  & 1.09  & 13007.48  & 43.20  & 1.43  & 11235.79  & 62.17  & 1.24  & 9719.64  & 200.69  & \textbf{1.07}  \\
          & 15    & \textbf{9250.84 } & 150.53  & 10124.14  & 14.79  & 1.09  & 13277.46  & 45.73  & 1.44  & 11482.59  & 66.84  & 1.24  & 9913.27  & 217.30  & \textbf{1.07}  \\
          & 20    & \textbf{9341.03 } & 154.13  & 10247.78  & 15.05  & 1.10  & 13417.59  & 48.32  & 1.44  & 11626.11  & 70.62  & 1.24  & 10015.98  & 233.78  & \textbf{1.07}  \\
        \bottomrule
    \end{tabular}%
\end{adjustbox}
\end{sidewaystable}%
\begin{sidewaystable}
\centering
\caption{S-shape strategy}
\label{Com_S}
	\begin{adjustbox}{scale=0.75,center}
    \begin{tabular}{crrrrrrrrrrrrrrr}
    \toprule
          &       & \multicolumn{2}{c}{AP2+VPG} & \multicolumn{3}{c}{ILS} & \multicolumn{3}{c}{Seed} & \multicolumn{3}{c}{CWI} & \multicolumn{3}{c}{CWII} \\
              \cmidrule(lr){3-4}
           \cmidrule(lr){5-7}
            \cmidrule(lr){8-10}
             \cmidrule(lr){11-13}
              \cmidrule(lr){14-16}
    \multicolumn{1}{l}{N} & \multicolumn{1}{l}{C} & \multicolumn{1}{l}{Obj} & \multicolumn{1}{l}{CPU} & \multicolumn{1}{l}{Obj} & \multicolumn{1}{l}{CPU} & \multicolumn{1}{l}{Obj ratio} & \multicolumn{1}{l}{Obj} & \multicolumn{1}{l}{CPU} & \multicolumn{1}{l}{Obj ratio} & \multicolumn{1}{l}{Obj} & \multicolumn{1}{l}{CPU} & \multicolumn{1}{l}{Obj ratio} & \multicolumn{1}{l}{Obj} & \multicolumn{1}{l}{CPU} & \multicolumn{1}{l}{Obj ratio} \\
    \hline
    \multirow{3}[0]{*}{480} & 10    & \textbf{5383.94 } & 109.50  & 6275.76  & 4.91  & 1.17  & 8430.08  & 4.31  & 1.57  & 6871.06  & 7.12  & 1.28  & 5979.76  & 22.52  & \textbf{1.11}  \\
          & 15    & \textbf{4834.33 } & 109.87  & 5672.26  & 5.84  & 1.17  & 7667.53  & 4.84  & 1.59  & 6239.29  & 6.99  & 1.29  & 5420.52  & 23.72  & \textbf{1.12}  \\
          & 20    & \textbf{4431.92 } & 110.03  & 5246.61  & 6.55  & 1.18  & 7031.87  & 5.30  & 1.59  & 5751.35  & 6.77  & 1.30  & 5019.21  & 24.90  & \textbf{1.13}  \\
    \cmidrule(lr){1-2}
    \multirow{3}[0]{*}{600} & 10    & \textbf{4993.35 } & 112.03  & 5903.82  & 6.29  & 1.18  & 7886.59  & 5.70  & 1.58  & 6345.16  & 39.42  & 1.27  & 5593.79  & 56.11  & \textbf{1.12}  \\
          & 15    & \textbf{5058.56 } & 113.26  & 6019.60  & 6.39  & 1.19  & 7998.79  & 6.26  & 1.58  & 6446.35  & 34.20  & 1.27  & 5657.56  & 55.75  & \textbf{1.12}  \\
          & 20    & \textbf{4968.91 } & 114.08  & 5947.26  & 6.72  & 1.20  & 7866.14  & 6.87  & 1.58  & 6349.91  & 30.75  & 1.28  & 5564.73  & 56.31  & \textbf{1.12}  \\
    \cmidrule(lr){1-2}
    \multirow{3}[0]{*}{720} & 10    & \textbf{5383.33 } & 116.05  & 6438.00  & 6.68  & 1.20  & 8529.64  & 7.35  & 1.58  & 6886.85  & 30.49  & 1.28  & 6007.23  & 61.49  & \textbf{1.12}  \\
          & 15    & \textbf{5511.58 } & 117.67  & 6595.83  & 6.86  & 1.20  & 8716.46  & 8.00  & 1.58  & 7041.03  & 30.02  & 1.28  & 6122.94  & 66.48  & \textbf{1.11}  \\
          & 20    & \textbf{5500.67 } & 119.01  & 6607.83  & 7.09  & 1.20  & 8686.68  & 8.73  & 1.58  & 7042.12  & 29.50  & 1.28  & 6104.12  & 71.18  & \textbf{1.11}  \\
    \cmidrule(lr){1-2}
    \multicolumn{1}{l}{Test(800)} & 16    & \textbf{6864.00 } & 141.78  & 8138.80  & 10.23  & 1.19  & 10771.00  & 16.14  & 1.57  & 9278.00  & 68.60  & 1.35  & 7484.00  & 144.19  & \textbf{1.09}  \\
    \cmidrule(lr){1-2}
    \multirow{3}[0]{*}{840} & 10    & \textbf{5866.16 } & 121.47  & 7058.51  & 7.09  & 1.20  & 9265.68  & 9.28  & 1.58  & 7488.64  & 31.67  & 1.28  & 6490.49  & 79.75  & \textbf{1.11}  \\
          & 15    & \textbf{6010.95 } & 123.78  & 7246.16  & 7.20  & 1.21  & 9474.76  & 9.99  & 1.58  & 7660.35  & 32.41  & 1.27  & 6632.17  & 88.82  & \textbf{1.10}  \\
          & 20    & \textbf{6030.68 } & 125.88  & 7292.67  & 7.49  & 1.21  & 9503.46  & 10.82  & 1.58  & 7693.20  & 33.20  & 1.28  & 6648.93  & 96.26  & \textbf{1.10}  \\
    \cmidrule(lr){1-2}
    \multirow{3}[0]{*}{960} & 10    & \textbf{6376.51 } & 129.30  & 7721.57  & 7.43  & 1.21  & 10044.63  & 11.44  & 1.58  & 8118.88  & 36.38  & 1.27  & 7000.15  & 107.90  & \textbf{1.10}  \\
          & 15    & \textbf{6526.80 } & 132.50  & 7911.98  & 7.60  & 1.21  & 10266.96  & 12.25  & 1.57  & 8311.73  & 39.01  & 1.27  & 7143.56  & 118.34  & \textbf{1.09}  \\
          & 20    & \textbf{6566.18 } & 135.38  & 7989.49  & 7.80  & 1.22  & 10327.14  & 13.18  & 1.57  & 8373.84  & 41.26  & 1.28  & 7180.76  & 129.00  & \textbf{1.09}  \\
    \cmidrule(lr){1-2}
    \multirow{3}[0]{*}{1080} & 10    & \textbf{6887.85 } & 139.17  & 8397.64  & 7.80  & 1.22  & 10841.29  & 13.86  & 1.57  & 8772.01  & 46.48  & 1.27  & 7514.32  & 241.88  & \textbf{1.09}  \\
          & 15    & \textbf{7041.82 } & 142.69  & 8594.43  & 7.92  & 1.22  & 11074.70  & 14.75  & 1.57  & 8976.42  & 51.47  & 1.27  & 7661.25  & 250.58  & \textbf{1.09}  \\
          & 20    & \textbf{7097.22 } & 146.01  & 8682.20  & 8.11  & 1.22  & 11149.21  & 15.77  & 1.57  & 9042.60  & 54.16  & 1.27  & 7711.17  & 259.30  & \textbf{1.09}  \\
    \cmidrule(lr){1-2}
    \multirow{3}[0]{*}{1200} & 10    & \textbf{7409.53 } & 150.28  & 9070.41  & 8.11  & 1.22  & 11641.02  & 16.51  & 1.57  & 9432.60  & 61.96  & 1.27  & 8031.03  & 271.89  & \textbf{1.08}  \\
          & 15    & \textbf{7561.31 } & 154.46  & 9280.85  & 8.21  & 1.23  & 11879.48  & 17.48  & 1.57  & 9633.64  & 65.90  & 1.27  & 8184.35  & 284.29  & \textbf{1.08}  \\
          & 20    & \textbf{7622.87 } & 158.41  & 9383.31  & 8.40  & 1.23  & 11970.60  & 18.61  & 1.57  & 9715.40  & 69.78  & 1.27  & 8243.50  & 296.97  & \textbf{1.08}  \\
        \bottomrule
    \end{tabular}%

\end{adjustbox}
\end{sidewaystable}%
Here, we compare our model with four commonly used heuristics for OBP. We plot line charts to visualize the domination of our approach on other algorithms under varying order sizes (See Figure \ref{Comfig}).

Table \ref{Com_R} and \ref{Com_S} show the test results under return and S-shape routing strategies. The Obj ratios are all comparison between AP2+VPG and other algorithms. In the comparison of solution quality in both tables, the AP2+VPG retained overall supremacy, which is about 10\% ahead of the second place, under all order numbers and capacity conditions. We can also find that The Seed and CWI are not very fit for this problem. The CWII remains the second-best heuristic in most cases, implying it is a highly powerful heuristic. This outcome is consistent with previous experiments published in the literature. The ILS is the third-best in most cases and slightly better than CWII in the scenarios of small order numbers and return strategy.

Taking computing time into consideration, the AP2+VPG grows slowly with the number of orders, from 110 seconds to around 150 seconds, implying that it is practical in real operation processes. This time cost is a good result for integrated branch-and-bound algorithms. The ILS is the fastest algorithm, but this does not compensate for its solution quality. The Seed and CWI algorithms show similar time costs of less than 60 seconds. The CWII shows a quadratic growth of computing time as the problem scale increases. This result is consistent with the description in \cite{de1999efficient}. 

\begin{figure}
	\centering
	\begin{minipage}{.5\textwidth}
		\centering
		\includegraphics[scale=0.55]{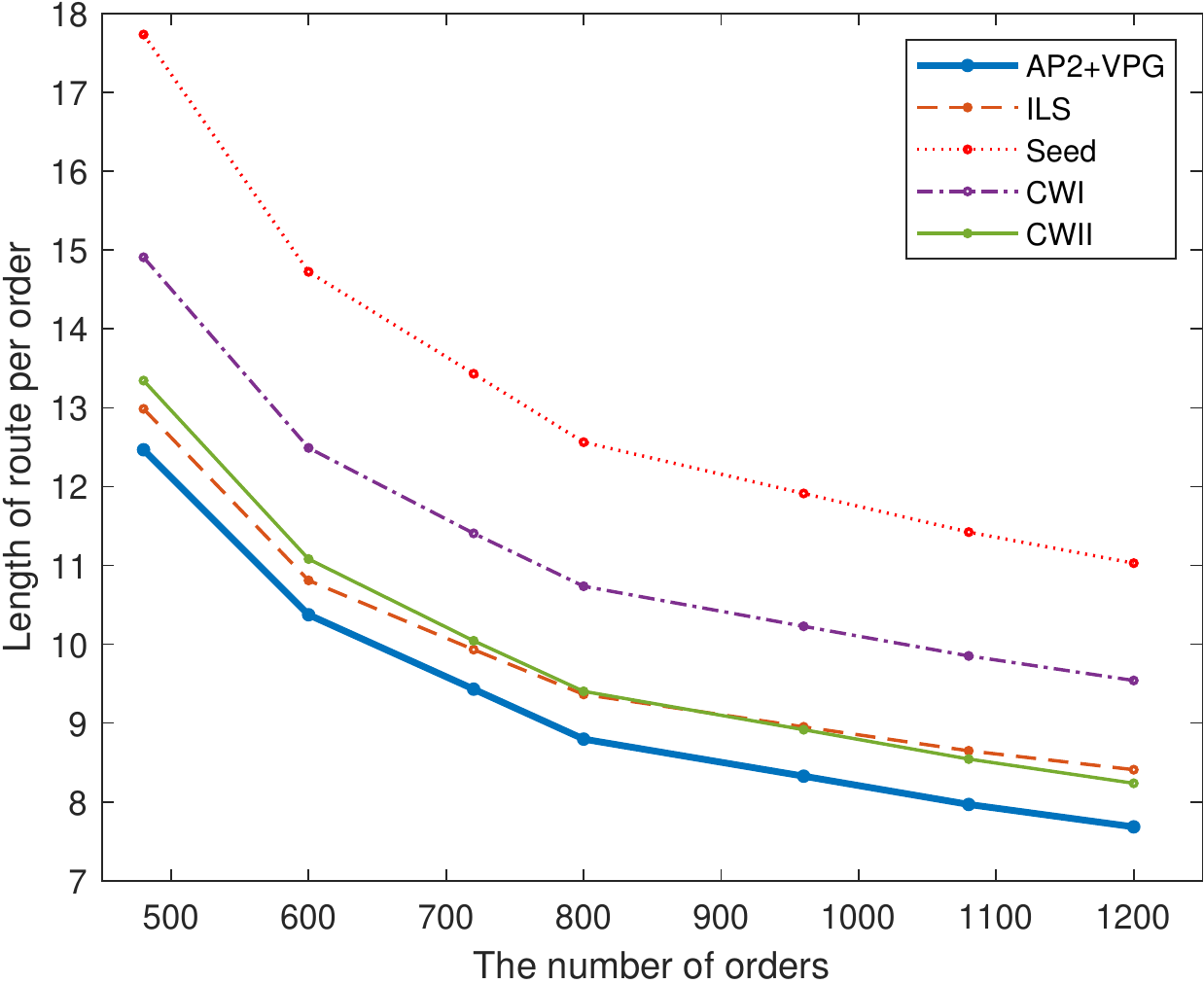}
		
	\end{minipage}%
	\begin{minipage}{.5\textwidth}
		\centering
		\includegraphics[scale=0.55]{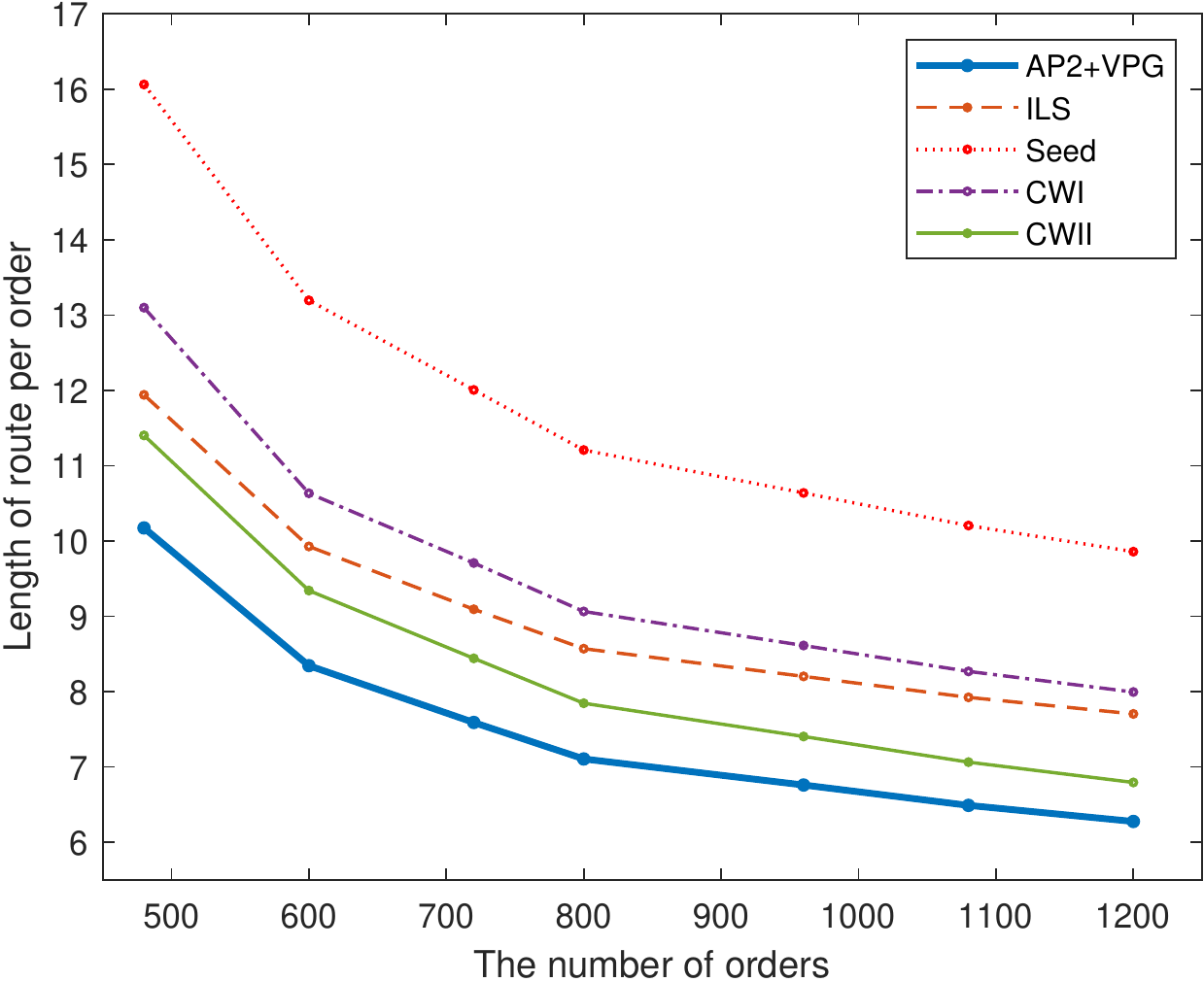}
		
	\end{minipage}
\caption{Average travel length per order comparison of Return routing strategy (left) and S-shape routing strategy (right).}
\label{Comfig}
\end{figure}

Finally, we can also find that the result of AP2 only is better than other algorithms, looking back to tables \ref{ComERS}. This also shows the effectiveness of our statistical analyses. Moreover, all algorithms perform better under the S-shape strategy, which means we can employ the S-shape strategy for routing in the real application of this problem.

\section{Conclusion}\label{con}
This paper proposes a new integrated order batching problem, the order batching and order packing problem provides the warehouse cost management in a supply chain with a more accurate perspective. To solve the new problem, this research offers a novel statistic-based framework to find an approximation MILP model. The approximation MILP model found in this framework gives high-quality solutions efficiently for medium and large-scale instances, compared with the proposed methods. Specifically, this method produces the most cost-saving order assignment in less than 150 seconds, which provides solid evidence that it can be implemented in software for practical operations. 

Future studies should concentrate on the strengthening and extension of this framework. There should be a general and non-problem-adaptive solution approach within this framework to find approximation MILP models. Moreover, this framework could be applied to other complex combinatorial optimization problems. In these problems, the complex problem settings and problem scale cause the exact MILP formulation unsolvable in a short time. This framework could provide an approximation MILP model that produces fast high-quality solutions.

\if0\blind{
\section*{Acknowledgements}
This work was funded by the National Nature Science Foundation of China under Grant No. 12071428 and 62111530247, and the Zhejiang Provincial Natural Science Foundation of China under Grant No. LZ20A010002.} \fi

\section*{Data availability statement}

The data that support the findings of this study are available on request from the corresponding author. The data are not publicly available due to that their containing information that could compromise the privacy of research participants.
\bibliographystyle{chicago}
\spacingset{1}
\bibliography{IISE-Trans}
	
\end{document}